\documentclass[11pt]{article}
\usepackage{amssymb,amsmath,amsfonts,amsthm,array}%
\usepackage{enumerate,bm,bbm,color}

\usepackage{tikz} 
\usetikzlibrary{patterns,positioning,arrows,shapes,trees}

\usepackage{graphicx}
\providecommand{\U}[1]{\protect\rule{.1in}{.1in}}
\numberwithin{equation}{section}

\newtheorem{theorem}{Theorem}[section]
\newtheorem{assumption}[theorem]{Assumption}
\newtheorem{lemma}[theorem]{Lemma}
\newtheorem{corollary}[theorem]{Corollary}
\newtheorem{proposition}[theorem]{Proposition}
\newtheorem{remark}[theorem]{Remark}
\newtheorem{example}[theorem]{Example}
\newtheorem{definition}[theorem]{Definition}

\def\<{\langle}
\def\>{\rangle}

\def\R{\mathbb{R}}
\def\T{\mathbb{T}}
\def\Z{\mathbb{Z}}

\def\cF{{\mathcal F}}

\def\cH{{\mathcal H}}

\def\cQ{{\mathcal Q}}

\def\cS{{\mathcal S}}

\def\cX{{\mathcal X}}

\def\mN{{\mathbb N}}

\def\mR{{\mathbb R}}

\def\bE{{\mathbf E}}

\def\bP{{\mathbf P}}

\def\geq{\geqslant}
\def\leq{\leqslant}

\def\l{\left}
\def\r{\right}
\def\<{\langle}
\def\>{\rangle}

\def\1{{\mathbf{1}}}
\def\p{\partial}
\def\d{\mathrm{d}}
\def\e{\mathrm{e}}

\def\div{\mathord{{\rm div}}}
\def\curl{\mathrm{curl}}

\usetikzlibrary{arrows,snakes,backgrounds,patterns}
\tikzstyle{notestyle} = [fill = yellow!35, rounded corners, text width=0.98\marginparwidth]
\tikzstyle{connectstyle} = [draw = red!50, thick]
\tikzstyle{my-yshift} = [yshift = 0.2\baselineskip]
\tikzstyle{my-xshift-right} = [xshift = -0.05\marginparwidth]
\tikzstyle{my-xshift-left} = [xshift = 0.04\marginparwidth]

\begin{document}
	
	\title{Quantitative estimates for SPDEs on the full space with\\ transport noise and $L^p$-initial data}
	
	\author{Dejun Luo\footnote{Email: luodj@amss.ac.cn. Key Laboratory of RCSDS, Academy of Mathematics and Systems Science, Chinese Academy of Sciences, Beijing 100190, China and School of Mathematical Sciences, University of Chinese Academy of Sciences, Beijing 100049, China} \quad
Bin Xie\footnote{Email: bxie@shinshu-u.ac.jp. Department of Mathematics, Faculty of Science, Shinshu University, Matsumoto, Nagano, 390-8621, Japan}
\quad Guohuan Zhao\footnote{Email: gzhao@amss.ac.cn. Institute of Applied Mathematics, Academy of Mathematics and Systems Science, CAS, Beijing, 100190, China}}
	
	\maketitle

\vspace{-20pt}
	
    \begin{abstract}
    For the stochastic linear transport equation with $L^p$-initial data ($1<p<2$) on the full space $\R^d$, we provide quantitative estimates, in negative Sobolev norms, between its solutions and that of the deterministic heat equation. Similar results are proved for the stochastic 2D Euler equations with transport noise.
    \end{abstract}

	%\tableofcontents
	
	
	
	\section{Introduction}

In the celebrated work \cite{FGP10}, Flandoli et al. proved that a $d$-dimensional Wiener process $W=W(t)$, in the form of transport noise, restores uniqueness of solutions to the linear transport equation with H\"older continuous drift terms. Since then, the theory of regularization by transport noise for stochastic partial differential equations (SPDEs) has attracted a lot of interests. However, as mentioned in \cite[Section 6.2]{FGP10}, the simple space-independent transport noise does not work for nonlinear equations, and one has to consider noises $W=W(t,x)$ with nontrivial space correlations. Let us mention that, in the last decade, there are numerous efforts trying to derive physically well-motivated noises in stochastic fluid dynamics equations, see \cite{Holm15} for variational arguments, \cite{FlanPap21, FlanPap22} for separation of scales and so on. It is by now widely accepted that multiplicative transport noise is appropriate for passive scalars and 2D Euler equations in vorticity form; however, for 3D fluid equations or passive magnetic fields, an extra stretching noise should be involved, see e.g. \cite{CFH19} and \cite[Section 1.2]{FlanLuo21}.

In recent years, motivated by Galeati's work \cite{Gal} on the scaling limit of stochastic linear transport equation to the deterministic heat equation, there have been intensive studies of similar limits of SPDEs with transport noise to the corresponding deterministic equations with extra dissipation term, see e.g. \cite{Agr24, Agr24b, CarLuon23, FGL21, FGL21b, FGL24, FlanLuo21, FLL24, GalLuo24, Luo21, Luo23, LuoTang23} and also \cite{BFL24, BFLT24, ButLuon24, FlanLuo24} for works concerning 3D fluid equations or passive magnetic fields; some of the results are also summarized in the lecture notes \cite{FlaLuon23}. Despite the different equations and topics treated in the above works, they have two points in common: the equations are studied on the torus $\T^d= \R^d/\Z^d$ and the solutions are assumed to be square integrable (with also $H^1$-regularity when dissipation terms are present in the SPDEs). The torus setting allows us to construct explicit orthonormal basis of the space $H= L^2_\sigma(\T^d, \R^d)$ of divergence free vector fields with zero spatial average. Not only does the basis leads to series expansion of $\R^d$-valued space-time noise $W=W(t,x)$, but it also simplifies many computations, e.g. in the estimates of stochastic convolutions.

The purpose of the current paper is twofold: extending the scaling limit results to SPDEs on the full space $\R^d$, and working in the framework of $L^p$-initial data with %$p\in (1, 2)$. 
$p<2$. In the full space setting, although there is no explicit basis of $L^2_\sigma(\R^d, \R^d)$, for general trace-class noise $W=W(t,x)$ on $\R^d$, Galeati and the first author of the present paper have shown that one can still expand the noise into an abstract series, see \cite[Section 2.1]{GalLuo} or Lemma \ref{lem:covariance-series-representation} below; such an expansion facilitates some expressions and calculations. Compared to the torus case where one often works with homogeneous spaces consisting of zero-average functions, on the full space $\R^d$ we have to be careful in the distinction of estimates in homogeneous and nonhomogeneous Sobolev spaces, see for instance Section 3 for estimates on stochastic convolutions. We mention that Agresti \cite{Agr24b} has applied scaling limit arguments beyond the $L^2$-setting to show global well-posedness by transport noise for stochastic hyperviscous Navier-Stokes equations, which are still considered on the torus $\T^3$.

In the remainder part of this section, we first give some definitions and notations frequently used in the sequel, then we describe the structure of vector-field valued noise $W=W(t,x)$, finally we state our quantitative estimates on the stochastic linear transport equations and stochastic 2D Euler equations.

For $x\in \R^d$, we write $\<x\>= (1+|x|^2)^{1/2}$. Denote by $\mathcal S$ the space of Schwartz test functions and
	$$
	\widehat{f}(\xi)=\cF(f)(\xi):= \frac1{(2\pi)^{d/2}} \int_{\mR^d} \e^{- {\rm i}\, x\cdot \xi} f(x)\, \d x ,\quad \xi\in \R^d
	$$
the Fourier transform of a function $f:\R^d\to \R^m$, $d,m\geq 1$. As usual, $L^p_x = L^p(\R^d)$ with $p\in [1,\infty]$ stands for the Lebesgue spaces with the norm $\|\cdot \|_{p}= \|\cdot \|_{L^p_x}$; for $\alpha\in \R$, we write $\dot H^\alpha_x= \dot H^\alpha(\R^d)$ for the usual homogeneous Sobolev space on $\R^d$, endowed with the semi-norm
  $$\|f \|_{\dot H^\alpha_x}= \bigg(\int_{\R^d} |\xi|^{2\alpha} |\widehat f(\xi)|^2\,\d\xi \bigg)^{1/2},$$
and the notation without dot is reserved for the corresponding inhomogeneous spaces, with a similar norm $\|f \|_{H^\alpha_x}$ replacing $|\xi|^{2\alpha}$ by $\<\xi \>^{2\alpha}$. We will use the same notations for spaces of vector fields on $\R^d$; there will be no confusion according to the context.

Given $T>0$ and $p,q\in [1,\infty]$, we write $L^p_t L^q_x$ for the time-dependent space $L^p(0,T; L^q(\R^d))$; in the same way, $C_t L^q_x$ and $C_t \dot H^\alpha_x$ stand for the spaces $C([0,T]; L^q(\R^d))$ and $C([0,T]; \dot H^\alpha(\R^d))$, respectively. Sometimes, we replace the subscript $t$ by $T$ to emphasize the length of time interval $[0,T]$. By $a\lesssim b$ we mean that there exists an unimportant constant $C>0$ such that $a\leq Cb$; if we want to stress the dependence of $C$ on some parameters $\alpha, p$, then we write $a\lesssim_{\alpha, p} b$. To avoid confusion with the vorticity variable $\omega$ in stochastic 2D Euler equations, we will write probability space as $(\Theta,\mathbb F, \bP)$ with generic element $\theta$.

	\subsection{Structure of the noise}\label{subsec-noise}

This part is mainly taken from \cite[Section 2.1]{GalLuo}. Let $W=W(t,x)$ be a centered Gaussian vector field defined on a filtered probability space $(\Theta,\mathbb F, \bP)$, possibly taking values in spaces of distributions. Assume that $W$ is white in time, coloured and divergence-free in the space variable;
we also assume it to be \emph{homogeneous} in space, so that its law is completely characterized by the covariance function $Q:\R^d \to \R^{d\times d}$ defined by
\begin{equation*}
	\bE[W(t,x)\otimes W(s,y)]= (t\wedge s) Q(x-y)
	\quad \forall\, t,s\in \R_+,\, x,y\in \R^d.
\end{equation*}
Here is our basic assumption on the structure of $Q$.

\begin{assumption}\label{ass:covariance-basic}
The covariance matrix $Q$ has Fourier transform $\widehat{Q}$ given by
\begin{equation}\label{eq:covariance-basic}
	\widehat{Q} (\xi)
	:=  g(\xi) P_\xi
	= g(\xi) \bigg(I_d- \frac{\xi\otimes \xi}{|\xi|^2} \bigg)
\end{equation}
for a non-negative radial function $g(\xi)=G(|\xi|)$ satisfying $g\in L^1 (\R^d) \cap L^\infty(\R^d) $; $I_d$ is the $d\times d$ unit matrix.
\end{assumption}

The matrix $P_\xi$ in \eqref{eq:covariance-basic} means the projection on the $(d-1)$-dimensional subspace of $\R^d$ orthogonal to $\xi\neq 0$; it corresponds to the divergence free property of the noise $W$.
Next, the radiality of $g$ implies that the noise is \emph{isotropic}: for any orthogonal matrix $R\in O(d)$, it holds
\begin{equation*}
	R\, Q(x) R^T = Q(R x), \quad R\, \widehat{Q}(\xi) R^T = \widehat{Q}(R \xi).
\end{equation*}
Since $g$ is real-valued, $Q$ is symmetric, namely $Q(x)=Q(-x)$. Moreover, one easily deduces from \eqref{eq:covariance-basic} that
\begin{equation}\label{eq:Q_0}
	Q(0)=2\kappa I_d, \quad \kappa:= \frac{d-1}{2d}\, \| g\|_{1}.
\end{equation}
Let $\cQ$ be the covariance operator associated to $Q$, acting on suitable $\R^d$-valued functions $f$:
\begin{equation*}
	(\cQ f)(x) := (Q\ast f)(x)= \int_{\R^d} Q(x-y) f(y)\,\d y;
\end{equation*}
equivalently, $\cQ$ can be defined as the Fourier multiplier $\cQ f= (2\pi)^{d/2} \mathcal{F}^{-1} \big(\widehat{Q} \widehat f\, \big)$. In the sequel we will often say that $W$ is a $\cQ$-Wiener process (or $Q$-Wiener process).

Let $\cQ^{1/2}$ be the square root of the covariance operator $\cQ$; we can define the Cameron-Martin space $\cH:=\cQ^{1/2}(L^2(\R^d;\R^d))$ of the noise $W(t,x)$. We refer to \cite[Section 2.1]{GalLuo} for more properties of the covariance function $Q$, the covariance operator $\cQ$ and the Cameron-Martin space $\cH$.

We now state tensorized representations of $Q$ and $W$, see \cite[Section 2.1]{GalLuo} for the proof.

\begin{lemma}\label{lem:covariance-series-representation}
Let $\{\sigma_k\}_{k\in\mathbb N}$ be any complete orthonormal system of $\cH$ made of smooth, divergence free functions. Then it holds
\begin{equation}\label{eq:covariance-series-representation}
	Q(x-y)=\sum_{k\in\mathbb N} \sigma_k(x) \otimes \sigma_k(y) \quad \forall\, x,y\in \R^d,
\end{equation}
where the series converges absolutely and uniformly on compact sets. Moreover, the noise admits the series representation
\begin{equation}\label{eq:noise-series-representation}
	W(t,x) = \sum_{k\in\mathbb N} \sigma_k(x) B^k_t,
\end{equation}
where $\{B^k \}_k$ is a family of independent standard Brownian motions defined by
\begin{equation*}\label{eq:chaos-expansion-finite-noises}
	B^k_t := \frac{\langle W_t, \cQ^{-1/2} \sigma_k\rangle}{\|\sigma_k \|_{L^2_x}}.
\end{equation*}
\end{lemma}

The following result will play a role in the estimates of stochastic convolutions.

\begin{proposition}\label{prop:sum-fourier-sigma}
    Under Assumption \ref{ass:covariance-basic}, the following identity holds in the sense of distribution:
    \begin{equation*}
        \sum_{k\in\mN} \widehat{\sigma_k}(\xi) \overline{\widehat{\sigma_k}(\eta)^{T}} = \widehat{Q}(\xi) \delta(\xi-\eta), \quad \xi, \eta\in \R^d,
    \end{equation*}
where $\widehat{\sigma_k}(\eta)^{T}$ is the transposition of the vector $\widehat{\sigma_k}(\eta)$, and the overline means complex conjugate; finally, $\delta(\xi-\eta)=1$ if $\xi=\eta$ and $0$ otherwise. 
%     Moreover, 
%     \begin{equation*}
%         \lim_{N\to\infty}  \sum_{k=N}^\infty \l( \int_{\mR^d}  \phi^{T}(\xi) \widehat{\sigma_k}(\xi) \d \xi \r)^2 =0, \quad \phi\in \cS(\mR^d; \mR^d). 
%     \end{equation*}
\end{proposition}
\begin{proof}
    Let $\phi, \psi\in \cS(\mR^d;\mR^d)$. Then 
    \begin{equation*}
        \begin{aligned}
            &\lim_{N\to\infty}  \iint_{\mR^d\times\mR^d}  \phi(\xi)^{T} \l(\sum_{k=0}^N\widehat{\sigma_k}(\xi)\overline{\widehat{\sigma_k}(\eta)^{T}}\r) \overline{\psi(\eta)} ~ \d \xi \d \eta\\
            &= \lim_{N\to\infty}  \iint_{\mR^d\times\mR^d}  \widehat{\phi}(x)^{T} \l(\sum_{k=0}^N \sigma_k(x) \sigma_k(y)^{T} \r) \overline{\widehat{\psi}(y)}~ \d x \d y\\
            &= \iint_{\mR^d\times\mR^d}  \widehat{\phi}(x)^{T}\, Q(x-y)\, \overline{\widehat{\psi}(y)}~ \d x \d y\\
            &= \int_{\mR^d} \phi(\xi)^T \widehat{Q}(\xi) \frac1{(2\pi)^{d/2}} \int_{\mR^d} \overline{\e^{{\rm i} y\cdot \xi} \widehat{\psi}(y)}~ \d y \d \xi  \\
            &= \int_{\mR^d} \phi(\xi)^T \widehat{Q}(\xi) \overline{\psi(\xi)}~ \d \xi, 
        \end{aligned}
    \end{equation*}
    where in the second identity we used the fact that $\sum_{k=0}^N \sigma_k(x) \sigma_k(y)^{T} \to Q(x-y)$ uniformly on any compact sets.  
\end{proof}

Finally, we give an example of a family of covariance functions $\{Q_\ell\}_{\ell\in (0,1)}$ whose Fourier transforms have constant $L^1$-norm, but their $L^q$-norms with $q>1$ vanish as $\ell\to 0$; these properties are needed for establishing convergence rates in the scaling limit results.

\begin{example}
		Fix $\lambda>0$. For any $\ell\in (0,1)$, assume that the function $g$ in \eqref{eq:covariance-basic} is given by
		\begin{equation}\label{eq:Kraichnan}
			g_\ell(\xi)= c_{d,\lambda}\, \ell^{-\lambda} |\xi |^{-(d+\lambda)} \1_{\{\ell^{-1}\leq |\xi|\leq 2\ell^{-1}\}},
		\end{equation}
where $c_{d,\lambda}$ is a normalizing constant such that $\|g_\ell\|_{L^1} =1$. However, for $q>1$, one has
$\|g_\ell \|_{L^q} \leq c_{d,\lambda,q}\, \ell^{d(q-1)/q}$ and, in particular, $\|g_\ell \|_{L^\infty}= c_{d,\lambda}\, \ell^{d}$; therefore, $\|g_\ell \|_{L^q}$ vanishes as $\ell\to 0$.

The choice \eqref{eq:Kraichnan} means that the noise $W(t,x)$ models turbulent small fluctuations at the scale $\ell\in (0,1)$. Thanks to the cut-off in frequency space, both the covariance function $Q_\ell$ and the corresponding noise are spatially smooth. 
	\end{example}

\subsection{Stochastic linear transport equations}
	
Let $W=W(t,x)$ be a centered $Q$-Wiener process defined on some filtered probability space  $(\Theta,\mathbb F, \bP)$, with a covariance function $Q$ satisfying Assumption \ref{ass:covariance-basic}. We consider the stochastic transport equation on $\mR^d$:
	\begin{equation}\label{eq:STE-S}
		\d f + \circ\d W \cdot \nabla f =0, \quad f|_{t=0} =f_0,
	\end{equation}
where $\circ \d$ stands for the Stratonovich stochastic differential. This form of noise is motivated by the Wong-Zakai principle, and it also preserves $L^p$-norms of solutions. Using the series expansion in \eqref{eq:noise-series-representation} of the noise $W$, equation \eqref{eq:STE-S} can be written more precisely as
  $$\d f +  \sum_k \sigma_k\cdot \nabla f\circ \d B^k_t =0. $$
For the sake of rigorous mathematical treatments, we need to transform the above equation to the It\^o form:
	\begin{equation}\label{eq:STE-I}
	\d f + \sum_k \sigma_k\cdot \nabla f\, \d B^k_t = \kappa \Delta f\,\d t,  \quad f|_{t=0} =f_0,
	\end{equation}
which can be deduced from equalities \eqref{eq:Q_0} and \eqref{eq:covariance-series-representation}. For simplicity, sometimes we also write \eqref{eq:STE-I} as
  $$\d f + \d W \cdot \nabla f = \kappa \Delta f\,\d t. $$
If $f_0\in L^p(\R^d)$ with $p\geq 1$, then the stochastic transport equation \eqref{eq:STE-I} admits a probabilistically strong solution $f$ satisfying $\|f_t\|_p \leq \|f_0\|_p$ $\bP$-a.s. for all $t\geq 0$; moreover, if in addition $f\in L^\infty_t L^1_x$, then the solution is also pathwise unique, see \cite[Theorem 1.3]{GalLuo}.

Inspired by \cite[Theorem 1.7]{FGL24}, we want to give a quantitative estimate between the solution of \eqref{eq:STE-I} and that of the heat equation
	\begin{equation}\label{eq:heat}
	    \partial_t \bar f = \kappa \Delta \bar f, \quad \bar{f} |_{t=0} =f_0.
	\end{equation}
Here is our first main result.

	\begin{theorem}\label{thm:STE}
		Assume $f_0\in L^p(\mR^d)$ for some $p\in (1,2]$.
\begin{itemize}
\item[\rm(1)] For any $\alpha\in \big(\frac{d}{2}, \frac{d}{2}+1 \big)$ and $q\geq 1$, it holds
\begin{equation*}
    \Big[ \bE \|f - \bar f\|_{C_T \dot H^{-\alpha}_x}^{q} \Big]^{1/q} \lesssim_{d, \alpha, q, \kappa, T} \|f_0\|_{p} \|\widehat{Q}\|_{p/(2-p)}^{1/2} .
\end{equation*}
\item[\rm(2)] For any $\alpha\in \big(d \big(\frac1p- \frac12 \big), \frac d2\big]$, $0<\varepsilon<\min\big\{1,  \alpha- d \big(\frac1p- \frac12 \big)\big\}$ and $q\geq 1$, we have
  $$
  \Big[ \bE \|f - \bar f\|_{C_T H^{-\alpha}_x}^{q} \Big]^{1/q}
  \lesssim_{d, \alpha,q,\kappa, T} \|f_0\|_{p} \|\widehat{Q}\|_{1}^{1/2- \varepsilon/(d+  4\varepsilon-2\alpha)}  \|\widehat{Q}\|_{p/(2-p)}^{\varepsilon/(d+ 4\varepsilon-2\alpha)} . $$
\end{itemize} 
\end{theorem}

The strategy of proofs, as in \cite{FGL24}, is rewriting both equations \eqref{eq:STE-I} and \eqref{eq:heat} in mild form, and noting that the difference $f_t - \bar f_t$ is nothing but the stochastic convolution
  $$Z_t= \int_0^t {\rm e}^{\kappa (t-s)\Delta}(\d W_s \cdot \nabla f_s) = \sum_k \int_0^t {\rm e}^{\kappa (t-s)\Delta}(\sigma_k \cdot \nabla f_s)\,\d B^k_s. $$
Therefore, the key step is to establish (maximal) estimates in suitable norms of $\{Z_t \}_{t\in [0,T]}$; this will be done in Section \ref{sec-stoch-convol}.

\subsection{Stochastic 2D Euler equations}
	
In this part, we consider the stochastic 2D Euler equations in vorticity form:
	\begin{equation}\label{eq:SEE-S}
		\d \omega + u\cdot\nabla\omega\, \d t + \circ \d W\cdot\nabla \omega=0,  \quad \omega|_{t=0} =\omega_0,
	\end{equation}
	where $\omega$ and $u$ are the fluid vorticity and velocity fields, respectively; and $W=W(t,x)$ is now a space-time Gaussian vector field on $\R^2$. The corresponding It\^o equation reads as
	\begin{equation}\label{eq:SEE-I}
		\d \omega + u\cdot\nabla\omega\, \d t + \d W\cdot\nabla \omega = \kappa \Delta\omega \,\d t,  \quad \omega|_{t=0} =\omega_0.
	\end{equation}
This model has received lots of attention in the last decade, for instance, in the setting of torus $\T^2$, see \cite{BFM16} for strong well-posedness with bounded vorticity $\omega_0$, \cite{FGL21} for weak existence with $L^2$-initial data, and \cite{FlanLuo19} for studies on white noise solutions and the associated Fokker-Planck equation. Recently, there are some well-posedness results on the full space $\R^2$, see \cite[Proposition 4.2]{GalLuo} for initial vorticity $\omega_0 \in L^1_x \cap L^p_x \cap \dot H^{-1}_x$ with $p\ge 2$. Coghi and Maurelli \cite{CogFla} studied stochastic 2D Euler equations \eqref{eq:SEE-S} driven by nonsmooth Kraichnan noise, i.e. the function $g$ in \eqref{eq:covariance-basic} is given by $g(\xi)=(1+ |\xi|^2)^{-(2 +\gamma)/2}$ for some $\gamma\in (0,2)$. Exploiting anomalous regularizing properties of the noise, they proved existence of weak solutions for all $\omega_0 \in \dot H^{-1}_x$; moreover, pathwise uniqueness of solutions was also shown for $\omega_0 \in L^1_x \cap L^p_x\cap \dot H^{-1}_x$ with suitable $\gamma\in (0,2)$ and $p>3/2$, see the short note \cite{JiaoLuo} for extension to the case $p>1$. In Section \ref{subsec-existence} below, we will follow the arguments in \cite{GalLuo} to show weak existence of \eqref{eq:SEE-I} for initial data in $L^p_x$ with $p> 4/3$, and general transport noise satisfying Assumption \ref{ass:covariance-basic}.

	We want to prove that the solution of \eqref{eq:SEE-I} is close to that of the $2$D Naiver-Stokes equation in vorticity form:
	\begin{equation}\label{eq:NS}
		\partial_t \bar{\omega} + \bar{u}\cdot\nabla \bar{\omega}= \kappa \Delta \bar{\omega}, \quad \quad \bar{\omega} |_{t=0} =\omega_0.
	\end{equation}
The following result is an analogue of \cite[Theorem 1.1]{FGL24}.

	\begin{theorem}\label{thm:SEE}
		Let $p\in \l(\sqrt{2}, 2 \r)$ and
      %and $\beta\in \l(\frac{2}{p}-1, \frac{2}{q}\r)$.
		$\alpha\in \big(2-p, 2-\frac{2}{p} \big)$.
  % and $\delta\in \big(0, \frac{\alpha}{2}+\frac{1}{2}- \frac{1}{p} \big)$. 
  Assume $\omega_0\in L^p_x$ and that Assumption \ref{ass:covariance-basic} holds. Let $\omega$ be a weak solution to \eqref{eq:SEE-I} and $\bar{\omega}$ the unique solution to \eqref{eq:NS} in $C_TL^p_x$. Then for any $q>p'$,  it holds that
		\begin{equation}
			\l[ \bE \|\omega-\bar{\omega}\|_{C_T \dot H^{-\alpha}_x}^q \r]^{1/q} \leq C_1 \|\omega_0 \|_{p} \|\widehat{Q}\|_{1}^{\theta/2} \|\widehat{Q}\|_{p/(2-p)}^{(1- \theta)/2},
   %C_1  \|\omega_0 \|_{p} \|\widehat{Q}\|_{\infty}^{\frac{\alpha}{2}+\frac{1}{2}-\frac{1}{p}-\delta} \exp\l( C_2 \|\omega_0 \|_{p}^q   \r),
		\end{equation}
        where $C_1$ depends on $\alpha, p,q, T$ and $\theta= p'(1-\alpha +\varepsilon)/2$ for some small $\varepsilon>0$. 
	\end{theorem}

We conclude this section by describing the structure of the paper. In Section \ref{sec-prelim}, we collect some preparatory results needed in the sequel. We prove in Section \ref{sec-stoch-convol} some maximal estimates on the stochastic convolutions which immediately imply Theorem \ref{thm:STE}. Then we provide in Section \ref{subsec-existence} a weak existence result for the stochastic 2D Euler equation \eqref{eq:SEE-I}, and finally we prove Theorem \ref{thm:SEE} in the last section.

	\section{Preparations}\label{sec-prelim}

    The following lemma is a classical result for regularity properties of heat semigroup $\{\e^{\kappa \Delta t}\}_{t\geq 0}$; we remark that the same results hold also in homogeneous Sobolev spaces $\dot H^{\alpha}$.
    
    \begin{lemma}\label{lem:HKE}
    	Let $u\in H^{\alpha}$, $\alpha\in\mathbb{R}$. Then
    	\begin{enumerate}[{\rm (i)}]
    	    \item for any $\rho\geq0$,
    		\begin{equation}\label{eq:hke1}
    			\|\e^{\kappa\Delta t} u
    			\|_{H^{\alpha+\rho}} \lesssim_{d, \kappa, \rho}  t^{-\rho/2} \|u\|_{H^{\alpha}};
    		\end{equation}
    		\item for any $\rho\in[0,2]$,
    		\begin{equation}\label{eq:hke2}
    			\|(I-\e^{\kappa \Delta t}) u\|_{H^{\alpha-\rho}}\lesssim_{d,\kappa,\rho} t^{\rho/2} \|u\|_{H^{\alpha}}.
    		\end{equation}
    	\end{enumerate}
    \end{lemma}

    For any $\alpha\in  (0,2)$, up to a multiplicative constant,
    $$
    \Lambda^{\alpha}u(x):= (-\Delta)^{\frac{\alpha}{2}}u(x) = -c_{d,\alpha}\, \mbox{P.V.} \int_{\mR^d}\frac{u(x+z)-u(x)}{|z|^{d+\alpha}}\d z.
    $$
    If we write
    \begin{equation}\label{eq:Gamma}
    	\Gamma_\alpha(u,v)(x):=c_{d,\alpha} \int_{\mR^d}(u(x+y)-u(x))(v(x+y)-v(x))\frac{\d y}{|y|^{d+\alpha}},
    \end{equation}
    then
    \begin{align}\label{eq:D-prod}
    	\Lambda^{\alpha}(uv)= (\Lambda^{\alpha}u) \, v+u\, \Lambda^{\alpha}v-\Gamma_\alpha(u,v).
    \end{align}

    The following result can be found in \cite[Theorem 2.36]{BCD}.
    \begin{lemma}
    	For any $p\in (1,\infty)$ and $\alpha\in (0,1)$, it holds that
    	\begin{equation}\label{eq:dif-Lp}
    		\|u(\cdot+y)-u(\cdot)\|_p \leq C |y|^\alpha \left\| \Lambda^{\alpha} u\right\|_p.
    	\end{equation}
    \end{lemma}

    We also need to use some fundamental properties of Lorentz spaces.  For more detailed explanations, we refer to the monograph \cite{Gra08}.  
    \begin{definition}
        For $1\leq p,q < \infty$, the Lorentz space $L^{p,q}(X,\mu )$ is the space of complex-valued measurable functions $f$ on a measure space $(X,\mu)$ such that the following quasi-norm
	\[
	    \|f\|_{L^{p,q}(X,\mu )}= p^{\frac{1}{q}} \l\{\int_0^\infty 
             t^{q-1} \l[\mu(\{x: |f(x)|>t\}) \r]^{\frac{q}{p}} ~ \d t \r\}^{\frac{1}{q}}
	\]
	is finite.  When $q=\infty$, we set 
	\[
           \|f\|_{{L^{{p,\infty }}(X,\mu )}} =\sup _{{t>0}} t \l[\mu ( \left\{x:|f(x)|>t\right\}) \r]^{\frac{1}{p}}.
        \]
        It is also conventional to set $L^{\infty,\infty}(X, \mu) = L^{\infty}(X, \mu)$.
    \end{definition}
	
	\begin{proposition}\label{prop:Lorentz}
		\begin{enumerate}[{\rm (i)}]
			\item (H\"older's inequality) Assume that $1\leq p, p_1, p_2\leq \infty$, $1\leq q, q_1, q_2\leq \infty$, then 
			\begin{equation}\label{Eq:Holder}
				\|fg\|_{L^{p,q}} \leq C \|f\|_{L^{p_1,q_1}} \|g\|_{L^{p_2, q_2}}, 
			\end{equation}
			where 
			\[
			    \frac{1}{p}=\frac{1}{p_1}+\frac{1}{p_2} ~\mbox{ and }~  \frac{1}{q}=\frac{1}{q_1}+\frac{1}{q_2}; 
			\]
			\item (Young's inequality) Assume that $1< p, p_1, p_2<\infty$, $1\leq q, q_1, q_2\leq \infty$, then 
			\begin{equation}\label{Eq:Young1}
				\|f*g\|_{L^{p,q}}\leq C \|f\|_{L^{p_1,q_1}} \|g\|_{L^{p_2, q_2}}, 
			\end{equation}
			where 
			\[
			    1+\frac{1}{p}=\frac{1}{p_1}+\frac{1}{p_2}, \quad  \frac{1}{q}\leq \frac{1}{q_1}+\frac{1}{q_2}; 
			\]
            and if $1<p_1, p_2<\infty, 1\leq q_1, q_2 \leq \infty$, then
            \begin{equation}\label{Eq:Young2}
                \|f * g\|_{L^{\infty}} \leq C\|f\|_{L^{p_1, q_1}}\|g\|_{L^{p_2, q_2}},
            \end{equation}
            where 
            \[
            \frac{1}{p_1}+\frac{1}{p_2}=1  ~\mbox{ and }~  \frac{1}{q_1}+\frac{1}{q_2}\geq 1. 
            \]
		\end{enumerate}
	\end{proposition}

\section{Quantitative estimates for stochastic transport equations} \label{sec-stoch-convol}

The first part of this section consists of some maximal estimates on stochastic convolutions in both homogeneous and nonhomogeneous spaces on $\R^d$, then they are applied to yield quantitative estimates between the solutions of stochastic transport equations and those of deterministic heat equation. 

\subsection{ Maximal estimates for stochastic convolutions}

Let 
\[
    P^\kappa_t= \e^{\kappa t \Delta}, \quad  h_t^{\kappa}=(4\kappa\pi t)^{-\frac{d}{2}} \e^{-\frac{|x|^2}{4\kappa t}}, \quad t\geq 0
\]
and $f: \Omega\times[0, T]\times\mR^d \to \mR$ be a predictable stochastic process. We define the stochastic convolution
\begin{equation*}
    \begin{aligned}
    Z_t:= Z_{0,t}, \quad Z_{s,t}:=\int_s^t P^{\kappa}_{t-r}(\d W_r \cdot \nabla f_r)= \int_s^t P^{\kappa}_{t-r} \nabla \cdot (f_r\, \d W_r)= \sum_{k,i} \int_s^t \p_i h^\kappa_{t-r} *(f_r\sigma_k^i) \, \d B^k_r,
    \end{aligned}
\end{equation*}
as long as the right-hand side term of the above formula can be defined. Here we have used the divergence free property of the noise. 

We will prove the following maximal estimates.
\begin{lemma}\label{lem:max}
    Assume $p\in [1,2]$ and that
    \begin{equation}\label{eq:f-preserve}
    \bP\mbox{-a.s.}, \quad  \sup_{t\in[0,T]}\|f_t\|_p\leq R.
    \end{equation}
Then for any $\alpha\in \big(\frac{d}{2}, \frac{d}{2}+1 \big)$ and $q\geq 1$, it holds
    \begin{equation}
        \bigg[ \bE \sup_{t\in [0,T]}\|Z_t\|_{\dot H^{-\alpha}}^{q} \bigg]^{1/q} \lesssim_{q,d, \alpha, \kappa, T} R\, \|\widehat{Q}\|_{p/(2-p)}^{1/2}.
    \end{equation}
\end{lemma}

\begin{proof}
  To avoid the complicated discussion of exchanging orders of limit and integration, we additionally assume that $f_t(\cdot,\theta)\in \cS$ for $(\d t\times \bP)$-a.e. $(t,\theta)\in [0,T]\times \Theta$, but the constants in our estimates below will not depend on the smoothness assumption on $f$.
  
  For $q\ge 1$, by the definition of $Z_t$, 
  \begin{equation}\label{proof-stoch-convol-1}
  \aligned
  \big[ \bE\|Z_t\|_{\dot H^{-\alpha}}^{2q} \big]^{1/q} 
  &= \bigg[ \bE \Big\| \sum_k \int_0^t P^\kappa_{t-r}\nabla\cdot (f_r \sigma_k )\, \d B^k_r \Big\|_{\dot H^{-\alpha}}^{2q} \bigg]^{1/q} \\ 
  &\lesssim_q \bigg[ \bE \Big( \sum_k \int_0^t \big\| P^\kappa_{t-r}\nabla\cdot (f_r \sigma_k ) \big\|_{\dot H^{-\alpha}}^2 \, \d r \Big)^q  \bigg]^{1/q} ,
  \endaligned
  \end{equation}
where in the second step we have used the Burkholder-Davis-Gundy inequality in the Hilbert space $\dot H^{-\alpha}$. Noting that 
  $$\aligned
  \big\| P^\kappa_{t-r}\nabla\cdot (f_r \sigma_k ) \big\|_{\dot H^{-\alpha}}^2 
  &= \int_{\R^d} |\xi|^{-2\alpha} \big|\mathcal F\big( P^\kappa_{t-r}\nabla\cdot (f_r \sigma_k ) \big)(\xi) \big|^2\, \d\xi \\
  &= \int_{\R^d} |\xi|^{-2\alpha} {\rm e}^{-2\kappa |\xi|^2(t-r) } \big|\xi\cdot \mathcal F(f_r \sigma_k )(\xi) \big|^2\, \d\xi \\
  &= (2\pi)^{-d} \int_{\R^d} |\xi|^{-2\alpha} {\rm e}^{-2\kappa |\xi|^2 (t-r)} \big|\xi\cdot (\widehat f_r\ast \widehat{\sigma_k})(\xi) \big|^2\, \d\xi ,
  \endaligned $$
therefore,
  $$\aligned
  \big[ \bE\|Z_t\|_{\dot H^{-\alpha}}^{2q} \big]^{1/q} 
  &\lesssim_{d,q} \bigg[ \bE \Big( \sum_k \int_0^t\! \int_{\R^d} |\xi|^{-2\alpha} {\rm e}^{-2\kappa |\xi|^2 (t-r)} \big|\xi\cdot (\widehat f_r\ast \widehat{\sigma_k})(\xi) \big|^2\, \d\xi \d r \Big)^q  \bigg]^{1/q} \\
  &= \bigg[ \bE \Big(  \int_{\R^d} |\xi|^{-2\alpha} \int_0^t {\rm e}^{-2\kappa |\xi|^2 (t-r)} \sum_k \big|\xi\cdot (\widehat f_r\ast \widehat{\sigma_k})(\xi) \big|^2\, \d r \d\xi \Big)^q \bigg]^{1/q}.
  \endaligned $$
By Proposition \ref{prop:sum-fourier-sigma}, we see that
  \begin{equation}\label{proof-stoch-convol-2}
  \aligned
      \sum_k |\xi\cdot (\widehat f_r\ast \widehat{\sigma_k})(\xi)|^2 &=\lim_{N\to\infty} \sum_{k=0}^N \iint_{\mR^d\times\mR^d} \widehat{f_r}(\xi-\eta)~ \xi^T \widehat{\sigma_k}(\eta) \overline{\widehat{\sigma_k}(\zeta)^T} \xi ~ \overline{\widehat{f_r} (\xi-\zeta)} ~\d \eta \d \zeta \\
      &= \int_{\mR^d} \xi^{T}\widehat{Q}(\eta) \xi\, |\widehat{f_r}(\xi-\eta)|^2~\d \eta. 
      \endaligned
  \end{equation}
%We have $|\xi\cdot (\widehat f_r\ast \widehat \sigma_k)(\xi)|^2 = \xi^T (\widehat f_r\ast \widehat \sigma_k)(\xi)\, \overline{(\widehat f_r\ast \widehat \sigma_k^T)(\xi)}\, \xi$. %where $\xi^T$ is the transposition of $\xi\in \mR^d$ and the overline means complex conjugate. 
%   Set 
%   $$\aligned
%   I&:= \sum_k (\widehat f_r\ast \widehat \sigma_k)(\xi)\, \overline{(\widehat f_r\ast \widehat \sigma_k)(\xi)^T} \\
%   &= \sum_k \bigg[\int \widehat f_r(\xi-\eta) \widehat\sigma_k(\eta) \,\d\eta\bigg] \bigg[\int \overline{\widehat f_r(\xi-\zeta)}\, \overline{\widehat\sigma_k(\zeta)^T} \,\d\zeta \bigg] \\
%   &= \int\int \widehat f_r(\xi-\eta) \overline{\widehat f_r(\xi-\zeta)} \sum_k \widehat\sigma_k(\eta) \overline{\widehat\sigma_k(\zeta)^T} \, \d\eta\d\zeta.
%   \endaligned $$
%   \cmt{$\wedge$}{Here includes some formal calculations}
% Some elementary computations lead to
%   $$\aligned
%   \sum_k \widehat\sigma_k(\eta) \overline{\widehat\sigma_k(\zeta)^T} &= \sum_k \int\int \sigma_k(x) \sigma_k(y)^T {\rm e}^{- {\rm i} (\eta\cdot x - \zeta\cdot y)}\, \d x\d y\\
%   &= \int\int Q(x-y) {\rm e}^{-{\rm i} (\eta\cdot x - \zeta\cdot y)}\, \d x\d y = \widehat{Q}(\eta) \delta_{\eta, \zeta},
%   \endaligned $$
% where $Q:\mR^d \to \mR^{d\times d}$ is the covariance matrix of the noise $W$, $\delta_{\eta, \zeta} =1$ if $\eta=\zeta$ and $0$ otherwise. Consequently,
%   $$I= \int |\widehat f_r(\xi-\eta)|^2 \widehat{Q}(\eta)\,\d\eta = \big(|\widehat f_r|^2 \ast \widehat{Q} \big)(\xi). $$
Plugging this into the inequality above, we get 
  \begin{equation*}
  \big[ \bE\|Z_t\|_{\dot H^{-\alpha}}^{2q} \big]^{1/q} \lesssim \bigg[ \bE \Big(  \int_{\R^d} |\xi|^{-2\alpha} \int_0^t {\rm e}^{-2\kappa |\xi|^2 (t-r)} \xi^T \big(|\widehat f_r|^2 \ast \widehat{Q} \big)(\xi) \xi \, \d r \d\xi \Big)^q \bigg]^{1/q} .
  \end{equation*}

To estimate the right-hand side, we apply Young's inequality: for some $a\geq 1$ to be determined later and any $\xi\in \mR^d$,
  $$\big|\big(|\widehat f_r|^2 \ast \widehat{Q} \big)(\xi) \big| \leq \big\| |\widehat f_r|^2 \big\|_{a} \|\widehat{Q}\|_{a'} = \|\widehat f_r\|_{2a}^2 \|\widehat{Q}\|_{a'} \leq \|f_r\|_{p}^2 \|\widehat{Q}\|_{a'},$$
where in the last step we have used that the Fourier transform $\cF: L^p \to L^{p'}$ is a bounded mapping ($p\in [1,2]$); therefore, we need $2a=p'$ which implies $a= p/(2(p-1))$ and $a'= p/(2-p)$. Combining this estimate with \eqref{eq:f-preserve} yields
  $$\aligned
   \big[ \bE\|Z_t\|_{\dot H^{-\alpha}}^{2q} \big]^{1/q} 
   &\lesssim \int_{\R^d} |\xi|^{-2\alpha+2} \int_0^t {\rm e}^{-2\kappa |\xi|^2 (t-r)} R^2 \|\widehat{Q}\|_{p/(2-p)} \, \d r \d\xi \\
   &=  R^2 \|\widehat{Q}\|_{p/(2-p)} \int_{\R^d} |\xi|^{-2\alpha} \frac1{2\kappa} \big(1- {\rm e}^{-2\kappa |\xi|^2 t} \big)\, \d\xi  \\
   &\le  R^2 \|\widehat{Q}\|_{p/(2-p)} \int_{\R^d} |\xi|^{-2\alpha} \big[(2\kappa)^{-1}\wedge (|\xi|^2 t) \big]\, \d\xi .
  \endaligned $$
Using spherical coordinates, one has
  $$\aligned
  & \int_{\mR^d} |\xi|^{-2\alpha} \big[(2 \kappa)^{-1}\wedge (|\xi|^2 t) \big] \,\d\xi \\
  &= c_d \int_0^\infty \rho^{-2\alpha} \big[(2 \kappa)^{-1}\wedge (\rho^2 t) \big] \rho^{d-1} \,\d\rho \\
  &\lesssim_{d, \alpha} t \int_0^{(2\kappa t)^{-1/2}} \rho^{-2\alpha + d+1} \,\d\rho + \kappa^{-1} \int_{(2 \kappa t)^{-1/2}}^\infty \rho^{-2\alpha + d-1} \,\d\rho \\
  &\lesssim_{d, \alpha} \kappa^{\alpha-d/2-1} t^{\alpha-d/2},
  \endaligned $$
thanks to the constraint $\alpha\in \big(\frac d2, 1+\frac d2 \big)$.  Substituting this estimate into the above inequality yields
  $$\big[ \bE\|Z_t\|_{\dot H^{-\alpha}}^{q} \big]^{1/q} \leq \big[ \bE\|Z_t\|_{\dot H^{-\alpha}}^{2q} \big]^{1/2q} \lesssim_{q,d, \alpha, \kappa} R \|\widehat{Q}\|_{p/(2-p)}^{1/2} t^{(2\alpha -d)/4}. $$

Recalling the definition of $Z_{s,t}$, in the same way, we can prove
  $$\big[ \bE\|Z_{s,t} \|_{\dot H^{-\alpha}}^{q} \big]^{1/q} \lesssim_{q,d, \alpha, \kappa} R \|\widehat{Q}\|_{p/(2-p)}^{1/2} |t-s|^{(2\alpha -d)/4}. $$
In fact, we have the simple identity: $Z_t = P^\kappa_{t-s} Z_s + Z_{s,t}$; therefore, taking a small $\varepsilon>0$ such that $\alpha- \varepsilon \in (\frac d2, \frac d2 +1)$, by Lemma \ref{lem:HKE}, it holds
  $$\aligned
  \|Z_t - Z_s\|_{\dot H^{-\alpha}} &\leq \|(I- P^\kappa_{t-s}) Z_s \|_{\dot H^{-\alpha}} + \|Z_{s,t}\|_{\dot H^{-\alpha}} \\
  &\lesssim_{d, \varepsilon } |t-s|^{\varepsilon/2}\|Z_s \|_{\dot H^{-\alpha+ \varepsilon}} + \|Z_{s,t}\|_{\dot H^{-\alpha}}.
  \endaligned $$
Combining with the above results, we obtain
  $$\aligned
  \big[ \bE\|Z_t - Z_s \|_{\dot H^{-\alpha}}^{q} \big]^{1/q} &\lesssim_{q,d, \alpha, \kappa} |t-s|^{\varepsilon/2} \big[ \bE\| Z_s \|_{\dot H^{-\alpha+\varepsilon}}^{q} \big]^{1/q} + \big[ \bE\|Z_{s,t} \|_{\dot H^{-\alpha}}^{q} \big]^{1/q} \\
  &\lesssim_T R \|\widehat{Q}\|_{p/(2-p)}^{1/2} \big(|t-s|^{\varepsilon/2} + |t-s|^{(2\alpha -d)/4}\big).
  \endaligned $$
As a result, for $\varepsilon_0 = \min\{\varepsilon/2, (2\alpha -d)/4\}>0$, we have
\begin{align}\label{esti-Z}
\bE\|Z_t - Z_s \|_{\dot H^{-\alpha}}^{q} \lesssim_{q,d, \alpha, \kappa, T} R^q \|\widehat{Q}\|_{p/(2-p)}^{q/2} |t-s|^{q \varepsilon_0}.
\end{align}
We can take $q$ big enough such that $q \varepsilon_0>1$, then Kolmogorov's continuity theorem implies that $\{Z_t\}_{t\in [0,T]}$ admits a version with continuous paths in $\dot H^{-\alpha}$, which will be denoted for simplicity with the same notation.

Moreover, recalling  the estimate \eqref{esti-Z} holds for any $q\geq 1$, we have  that
\begin{align*}
 \bE\|Z\|_{W^{s, q}(0,T; \dot H^{-\alpha})}^q \lesssim_{q,d, \alpha, \kappa, T} R^q \|\widehat{Q}\|_{p/(2-p)}^{q/2}
\end{align*}
for any  $0<s<\varepsilon_0$, which,  by the Sobolev  embedding theorem for large $q\ge 1$,  implies that
\begin{align*}
 \bE\|Z\|_{C^{0, \gamma}([0,T]; \dot H^{-\alpha})}^q\lesssim_{q,d, \alpha,\kappa, T} R^q \|\widehat{Q}\|_{p/(2-p)}^{q/2}
\end{align*}
holds for any $0 < \gamma <\varepsilon_0.$
In particular, it holds
  $$\bE\bigg[\sup_{t\in [0,T]} \|Z_t \|_{\dot H^{-\alpha}}^{q} \bigg] \lesssim_{q,d, \alpha,\kappa, T} R^q \|\widehat{Q}\|_{p/(2-p)}^{q/2}, $$
which leads to the desired estimate.
\end{proof}

Next we give a maximal estimate of the stochastic convolution in nonhomogeneous Sobolev spaces with parameter close to the critical value $\frac{d}{p}-\frac{d}{2}$.

\begin{lemma}\label{lem:max-2}
    	Assume \eqref{eq:f-preserve} holds for some $p\in (1,2]$.
    	 Then for any $\frac{d}{p}-\frac{d}{2} < \alpha \leq \frac{d}{2}$, one has
    		\begin{equation}
    			\bigg[ \bE \sup_{t\in [0,T]}\|Z_t\|_{H^{-\alpha}}^{q} \bigg]^{1/q}
    			\lesssim_{ d, \alpha,q,\kappa, T} R\, \|\widehat{Q}\|_{1}^{1/2}.
    		\end{equation}
    \end{lemma}

\begin{proof}
We follow the idea of proof of \cite[Lemma 2.5]{FGL24}. Again we additionally assume that $f_t(\cdot,\theta)\in \cS$ for $(\d t\times \bP)$-a.e. $(t,\theta)\in [0,T]\times \Theta$. By the definition of $Z_t$ and the Burkholder-Davis-Gundy inequality in Hilbert spaces, we have
  $$\aligned
  \big[ \bE\|Z_t \|_{H^{-\alpha}}^{2q} \big]^{1/q} &= \bigg[ \bE \Big\|\sum_k \int_0^t P^\kappa_{t-s}\nabla\cdot (f_s \sigma_k) \,\d B^k_s \Big\|_{H^{-\alpha}}^{2q} \bigg]^{1/q}\\
  &\lesssim_q \bigg[ \bE \Big(\sum_k \int_0^t \| P^\kappa_{t-s} \nabla\cdot (f_s \sigma_k)\|_{H^{-\alpha}}^2 \,\d s \Big)^q \bigg]^{1/q} \\
  &\lesssim \bigg[ \bE \Big(\sum_k \int_0^t \frac1{(\kappa |t-s|)^{1-\varepsilon}} \|\nabla\cdot (f_s \sigma_k)\|_{H^{-\alpha-1 +\varepsilon}}^2 \,\d s \Big)^q \bigg]^{1/q},
  \endaligned $$
where $\varepsilon\in (0,1)$ is a small number such that $\alpha-\varepsilon > d(\frac1p -\frac12)$ and we have used Lemma \ref{lem:HKE} in the last step. Moreover,
  \begin{equation}\label{lem:max-2-1}
  \aligned
  \big[ \bE\|Z_t \|_{H^{-\alpha}}^{2q} \big]^{1/q} &\lesssim \kappa^{\varepsilon-1} \bigg[ \bE \Big(\sum_k \int_0^t \frac1{|t-s|^{1-\varepsilon}} \|f_s \sigma_k \|_{H^{-\alpha +\varepsilon}}^2 \,\d s \Big)^q \bigg]^{1/q}.
  \endaligned
  \end{equation}

It remains to estimate $\sum_k \|f_s \sigma_k \|_{H^{-\alpha +\varepsilon}}^2$; using the notation $\<\xi \>= (1+|\xi|^2)^{1/2}$ we have
  $$\sum_k \|f_s \sigma_k \|_{H^{-\alpha +\varepsilon}}^2 = \sum_k \int_{\mR^d} \<\xi \>^{-2(\alpha-\varepsilon)} \big|\widehat{f_s \sigma_k} (\xi)\big|^2 \,\d\xi = \frac1{(2\pi)^{d}} \int_{\mR^d} \<\xi \>^{-2(\alpha-\varepsilon)} \sum_k \big|(\widehat f_s\ast \widehat \sigma_k) (\xi)\big|^2 \,\d\xi.  $$
%Similarly to the treatment of $I$ in the proof of Lemma \ref{lem:max}, it holds
Similarly to the proof of Lemma \ref{lem:max}, it holds
  $$\aligned
  \sum_k \big|(\widehat f_s\ast \widehat \sigma_k) (\xi)\big|^2
  &= \sum_k \iint_{\mR^d\times\mR^d} \widehat f_s(\xi-\eta) \overline{\widehat f_s(\xi-\zeta)}\, \widehat \sigma_k(\eta)\cdot \overline{\widehat \sigma_k(\zeta)} \,\d\eta\d \zeta \\
  &= \int_{\mR^d} |\widehat f_s(\xi-\eta)|^2\, {\rm Tr}(\widehat{Q}(\eta))\,\d\eta .
  \endaligned $$ 
Therefore,
  $$\aligned
  \sum_k \|f_s \sigma_k \|_{H^{-\alpha +\varepsilon}}^2 &\lesssim_d  \int_{\mR^d} \<\xi \>^{-2(\alpha-\varepsilon)} \big(|\widehat f_s|^2 \ast {\rm Tr}(\widehat{Q})\big)(\xi)\,\d\xi \\
  &\leq \big\| \<\cdot \>^{-2(\alpha-\varepsilon)} \big\|_b \big\| |\widehat f_s|^2 \ast {\rm Tr}(\widehat{Q}) \big\|_{b'} \\
  &\leq \big\| \<\cdot \>^{-2(\alpha-\varepsilon)} \big\|_b \big\| |\widehat f_s|^2 \big\|_{b'} \|{\rm Tr}(\widehat{Q}) \|_{1},
  \endaligned $$
where $\frac1b + \frac1{b'}=1$ and the last step is due to Young's inequality. We want that $\big\| |\widehat f_s|^2 \big\|_{b'} = \| \widehat f_s \|_{2b'}^2 \leq \|f_s \|_{p}^2$, namely, $2b'$ should be the conjugate number of $p$, thus $b=p/(2-p)$. Note that, in this case,
  $$\aligned
  \big\| \<\cdot \>^{-2(\alpha-\varepsilon)} \big\|_b &= \bigg[\int \<\xi \>^{-2(\alpha-\varepsilon)p/(2-p)} \,\d\xi \bigg]^{(2-p)/p} \\
  &\lesssim_{d,p} \bigg[\int_0^\infty (1+\rho^2)^{-(\alpha-\varepsilon)p/(2-p)} \rho^{d-1} \,\d\rho \bigg]^{(2-p)/p};
  \endaligned $$
the last integral is finite because $\frac{2p}{2-p}(\alpha-\varepsilon) -(d-1)>1$ which is equivalent to $\alpha-\varepsilon > d(\frac1p -\frac12)$. Summarizing these computations, we arrive at
  $$\sum_k \|f_s \sigma_k \|_{H^{-\alpha +\varepsilon}}^2 \lesssim_{d,p,\varepsilon} \|f_s \|_{p}^2 \|\widehat{Q}\|_{1}. $$
Substituting this estimate into \eqref{lem:max-2-1} and using \eqref{eq:f-preserve} yield
  $$\aligned
  \big[ \bE\|Z_t \|_{H^{-\alpha}}^{2q} \big]^{1/q} &\lesssim \kappa^{\varepsilon-1} \bigg[ \bE \Big( \int_0^t \frac1{|t-s|^{1-\varepsilon}} \|f_s \|_{p}^2 \|\widehat{Q}\|_{1} \,\d s \Big)^q \bigg]^{1/q} \\
  &\lesssim_\kappa R^2 \|\widehat{Q}\|_{1} \int_0^t \frac1{|t-s|^{1-\varepsilon}} \,\d s\\
  & \lesssim_\varepsilon R^2 \|\widehat{Q}\|_{1}\, t^\varepsilon.
  \endaligned $$
As a result,
  $$\big[ \bE\|Z_t \|_{H^{-\alpha}}^{q} \big]^{1/q} \leq \big[ \bE\|Z_t \|_{H^{-\alpha}}^{2q} \big]^{1/2q} \lesssim R \|\widehat{Q}\|_{1}^{1/2}\, t^{\varepsilon/2}. $$

In the same way, we have
  $$\big[ \bE\|Z_{s,t} \|_{H^{-\alpha}}^{q} \big]^{1/q} \lesssim R \|\widehat{Q}\|_{1}^{1/2}\, |t-s|^{\varepsilon/2}.$$
The remaining proofs are similar to that of Lemma \ref{lem:max} so we omit them.
\end{proof}

The following estimate on inhomogeneous Sobolev norms follows from interpolation.

\begin{corollary}\label{cor-stoch-convol-1}
Assume \eqref{eq:f-preserve} holds for some $p\in (1,2].$ For any
$\alpha\in \big(d(\frac1p- \frac12), \frac d2\big]$
 and $0<\varepsilon< \min \big\{1, \ \alpha- d(\frac1p- \frac12) \big\}$, we have
$$
\bigg[ \bE \sup_{t\in [0,T]}\|Z_t\|_{H^{-\alpha}}^{q}\bigg]^{1/q} \lesssim_{d, \alpha,q, \kappa, T} R \|\widehat{Q}\|_{1}^{1/2- \varepsilon/(d +4\varepsilon -2\alpha)}  \|\widehat{Q}\|_{p/(2-p)}^{\varepsilon/(d+4\varepsilon -2\alpha)} . $$
\end{corollary}

\begin{proof}
Since $\alpha-\varepsilon > d(\frac1p- \frac12)$, we can apply Lemma \ref{lem:max-2} to get
  $$\bigg[ \bE \sup_{t\in [0,T]}\|Z_t\|_{H^{-\alpha +\varepsilon}}^{q} \bigg]^{1/q} \lesssim_{d, \alpha,q, \kappa, T} R \|\widehat{Q}\|_{1}^{1/2} .$$
Next, note that the inhomogeneous Sobolev norms of negative order are weaker than the corresponding homogeneous ones; applying Lemma \ref{lem:max} with $\alpha= \frac d2+\varepsilon$ gives us
  $$\bigg[ \bE \sup_{t\in [0,T]}\|Z_t\|_{H^{-d/2 -\varepsilon}}^{q} \bigg]^{1/q} \leq \bigg[ \bE \sup_{t\in [0,T]}\|Z_t\|_{\dot H^{-d/2 -\varepsilon}}^{q} \bigg]^{1/q} \lesssim_{d, \alpha,q, \kappa, T} R \|\widehat{Q}\|_{p/(2-p)}^{1/2} .$$
Let $\theta=2\varepsilon/(d +4\varepsilon -2\alpha)$, we have
 $$
 -\alpha = (1- \theta )(-\alpha +\varepsilon) +\theta \Big(- \frac d2 -\varepsilon \Big),
 $$
and thus by interpolation,
  $$\|Z_t\|_{H^{-\alpha}} \leq \|Z_t\|_{H^{-\alpha+ \varepsilon}}^{1-\theta} \|Z_t\|_{H^{-d/2 -\varepsilon}}^{\theta}. $$
Therefore, by H\"older's inequality,
  $$\aligned
  \bE \sup_{t\in [0,T]}\|Z_t\|_{H^{-\alpha}}^{q}
  &\leq \bigg[\bE \sup_{t\in [0,T]}\|Z_t\|_{H^{-\alpha +\varepsilon}}^{q} \bigg]^{1- \theta} \bigg[ \bE \sup_{t\in [0,T]}\|Z_t\|_{H^{-d/2-\varepsilon}}^{q} \bigg]^{\theta}\\
  &\lesssim R^q \|\widehat{Q}\|_{1}^{q(1- \theta)/2}  \|\widehat{Q}\|_{p/(2-p)}^{q \theta/2},
  \endaligned $$
which leads to the desired result.
\end{proof}

We can also derive estimates on the homogeneous Sobolev norms of negative order, using a slightly more complicated interpolation; note that the range of $\alpha$ below is not empty since $p>1$.

\begin{corollary}\label{cor-stoch-convol-2}
Assume \eqref{eq:f-preserve} holds for some $p\in (1,2]$. For any $\alpha\in \big(\frac{d(d+2)(2-p)}{2(2p + d(2-p))}, \frac d2\big]$, we have
  $$\bigg[ \bE \sup_{t\in [0,T]}\|Z_t\|_{\dot H^{-\alpha}}^{q}\bigg]^{1/q} \lesssim_{d, \alpha,q,\kappa, T} R \|\widehat{Q}\|_{1}^{\theta/2} \|\widehat{Q}\|_{p/(2-p)}^{(1-\theta)/2} , $$
where $\theta= p'(\frac d2+\varepsilon -\alpha)/d= p(\frac d2+\varepsilon -\alpha)/d(p-1)$ for some small $\varepsilon>0$. 
\end{corollary}

\begin{proof}
We start by proving the following interpolation estimate: let $0<\gamma<\alpha <\beta$ and $\theta\in (0,1)$ such that $\alpha= \theta\gamma+ (1-\theta)\beta$, then for any $f\in H^{-\gamma} \cap \dot H^{-\beta-\frac\theta{1-\theta}\gamma}$, it holds
  \begin{equation}\label{interpolation-mixed}
  \|f\|_{\dot H^{-\alpha}} \lesssim_{\theta,\gamma} \|f\|_{H^{-\gamma}}^{\theta} \Big[\|f\|_{\dot H^{ -\beta -\frac\theta{1-\theta}\gamma}} + \|f\|_{\dot H^{-\beta}} \Big]^{1-\theta}.
  \end{equation}
Indeed, by definition,
  $$\aligned
  \|f\|_{\dot H^{-\alpha}}^2 & = \int_{\mR^d} \frac{|\widehat f(\xi)|^2}{|\xi|^{2\alpha}}\,\d\xi = \int_{\mR^d} \frac{|\widehat f(\xi)|^{2\theta}}{|\xi|^{2\theta\gamma}} \frac{|\widehat f(\xi)|^{2(1-\theta)}}{|\xi|^{2(1-\theta)\beta}} \,\d\xi \\
  &= \int_{\mR^d} \frac{|\widehat f(\xi)|^{2\theta}}{(1+|\xi|^2)^{\theta\gamma}} \frac{(1+|\xi|^2)^{\theta \gamma}}{|\xi|^{2\theta \gamma}} \frac{|\widehat f(\xi)|^{2(1-\theta)}}{|\xi|^{2(1-\theta)\beta}} \,\d\xi;
  \endaligned$$
H\"older's inequality leads to
  $$\aligned
  \|f\|_{\dot H^{-\alpha}}^2 &\leq \bigg[\int_{\mR^d} \frac{|\widehat f(\xi)|^2}{(1+|\xi|^2)^{\gamma}}\,\d\xi \bigg]^\theta \bigg[\int_{\mR^d} \Big(\frac{1+|\xi|^2}{|\xi|^2}\Big)^{\frac{\theta\gamma}{1-\theta}} \frac{|\widehat f(\xi)|^2}{|\xi|^{2\beta}} \,\d\xi \bigg]^{1-\theta} =: \|f\|_{H^{-\gamma}}^{2\theta} J^{1-\theta},
  \endaligned $$
where $J$ denotes the integral in the second square bracket. We have
  $$\aligned J&= \bigg(\int_{|\xi|\leq 1} + \int_{|\xi|> 1}\bigg) \Big(\frac{1+|\xi|^2}{|\xi|^2} \Big)^{\frac{\theta \gamma}{1-\theta}} \frac{|\widehat f(\xi)|^2}{|\xi|^{2\beta}} \,\d\xi \\
  &\leq \int_{|\xi|\leq 1} \frac{2^{\theta\gamma/(1-\theta)}}{|\xi|^{2\beta+ 2\gamma\theta/(1-\theta)}} |\widehat f(\xi)|^2 \,\d\xi + \int_{|\xi|> 1} 2^{\theta\gamma/(1-\theta)} \frac{|\widehat f(\xi)|^2}{|\xi|^{2\beta}} \,\d\xi \\
  &\lesssim_{\theta,\gamma} \|f\|_{\dot H^{ -\beta -\frac\theta{1-\theta}\gamma}}^2 + \|f\|_{\dot H^{-\beta}}^2.
  \endaligned  $$
Summing up these estimates gives rise to \eqref{interpolation-mixed}.

In order to apply the interpolation inequality \eqref{interpolation-mixed} and Lemma \ref{lem:max}, we claim that, for $\alpha$ given as in the statement, we can take $\gamma\in \big(d(\frac1p -\frac12), \alpha\big)$ and $\beta\in \big(\frac d2, \frac d2+1\big)$ such that
  \begin{equation}\label{cor-stoch-convol-2-1}
  \beta+ \frac\theta{1-\theta}\gamma < \frac d2+1,
  \end{equation}
where $\theta= \frac{\beta-\alpha}{\beta-\gamma}$. Indeed, we first note that $p\in (1,2]$ implies $d(\frac1p -\frac12) \leq \frac{d(d+2)(2-p)}{2(2p + d(2-p))}$, and thus the range for $\gamma$ is not empty; next,  inequality \eqref{cor-stoch-convol-2-1} is equivalent to
  $$\frac{\beta-\alpha}{\beta-\gamma} = \theta < \frac{\frac d2+1-\beta}{\frac d2+1-\beta +\gamma},$$
which can be rewritten as
  $$\alpha> \frac{\frac d2+1}{\frac d2+1-\beta +\gamma} \gamma. $$
This is possible because we can take $\gamma$ and $\beta$ which are sufficiently close to $d(\frac1p -\frac12)$ and $\frac d2$, respectively, thus the above inequality follows from
  $$\alpha > \frac{\frac d2+1}{1 +d(\frac1p -\frac12)} d\Big(\frac1p -\frac12 \Big); $$
the latter is nothing but $\alpha> \frac{d(d+2)(2-p)}{2(2p + d(2-p))}$.

Inspired by the above arguments, we take $\varepsilon>0$ small enough and let $\gamma= d(\frac1p -\frac12)+ \varepsilon$, $\beta= \frac d2+\varepsilon$, then \eqref{cor-stoch-convol-2-1} is verified. By Lemma \ref{lem:max}, we have
  $$\bigg[ \bE \sup_{t\in [0,T]}\|Z_t\|_{\dot H^{ -\beta -\frac\theta{1-\theta}\gamma}}^{q} \bigg]^{1/q} \lesssim_{d, \alpha, q, \kappa, T} \|\widehat{Q}\|_{p/(2-p)}^{1/2} R. $$
We also have
  $$\bigg[ \bE \sup_{t\in [0,T]}\|Z_t\|_{\dot H^{ -\beta}}^{q} \bigg]^{1/q} \lesssim_{d, \alpha, q, \kappa, T} \|\widehat{Q}\|_{p/(2-p)}^{1/2} R $$
and the inhomogeneous estimate
  $$\bigg[ \bE \sup_{t\in [0,T]}\|Z_t\|_{ H^{ -\gamma}}^{q} \bigg]^{1/q} \lesssim_{d, \alpha, q, \kappa, T} \|\widehat{Q}\|_{1}^{1/2} R. $$
Combining these estimates with \eqref{interpolation-mixed}, for $\theta= p'(\frac d2+\varepsilon -\alpha)/d$, we have
  $$\aligned
  \bE \sup_{t\in [0,T]}\|Z_t\|_{\dot H^{-\alpha}}^{q} &\lesssim \bigg[\bE \sup_{t\in [0,T]}\|Z_t \|_{H^{-\gamma}}^{q} \bigg]^\theta \bigg[\bE \sup_{t\in [0,T]}\|Z_t \|_{ \dot H^{ -\beta -\frac\theta{1-\theta}\gamma}}^{q} + \bE \sup_{t\in [0,T]}\|Z_t \|_{ \dot H^{ -\beta}}^{q} \bigg]^{1-\theta} \\
  &\lesssim \Big[R^q \|\widehat{Q}\|_{1}^{q/2}\Big]^\theta \Big[R^q \|\widehat{Q}\|_{p/(2-p)}^{q/2} \Big]^{1-\theta} \\
  &= R^q \|\widehat{Q}\|_{1}^{q\theta/2} \|\widehat{Q}\|_{p/(2-p)}^{q(1-\theta)/2}.
  \endaligned $$
This leads to the desired estimate.
\end{proof}
% added by GZ
\begin{remark}
    Assume \eqref{eq:f-preserve} holds for some $p\in (1,2]$. If additionally 
    \begin{equation}\label{eq:Z-Lp}
        \sup_{t\in [0,T]} \l[\bE \|Z_t\|_p^q\r]^{\frac{1}{q}}\leq A \|f_0\|_p, \quad q\geq 1. 
    \end{equation}
    Then for any $1<p\leq 2, ~ \frac{d}{p}-\frac{d}{2} < \alpha \leq \frac{d}{2} ~\mbox{ and }~  0<\delta< \frac{\alpha}{d}+\frac{1}{2}- \frac{1}{p}$,
    \begin{equation*}
        \l[ \bE \sup_{t\in [0,T]}\|Z_t\|_{\dot H^{-\alpha}}^{q}\r]^{\frac{1}{q}} \lesssim_{d, \alpha, \delta, q, A} \|\widehat{Q}\|_\infty^{\frac{\alpha}{d}+\frac{1}{2}-\frac{1}{p}-\delta } \|f_0\|_p. 
    \end{equation*}
\end{remark}

    \subsection{Proof for Theorem \ref{thm:STE}}
    Rewriting \eqref{eq:STE-S} and \eqref{eq:heat} in mild form, and taking difference, we obtain
    \begin{equation*}
    	\begin{aligned}
    		f_t-\bar{f}_t & =\int_0^t P^\kappa_{t-r} \l( \nabla f_r \cdot  \d W_r \r)= \int_0^t P^\kappa_{t-r} \nabla \cdot (f_r \,\d W_r) = Z_t.
    	\end{aligned}
    \end{equation*}
Recall that if $f_0\in L^p(\R^d)$, then the stochastic transport equation \eqref{eq:STE-I} admits a probabilistically strong solution $f$ satisfying, $\bP$-a.s., $\|f_t\|_p \leq \|f_0\|_p$ for all $t\geq 0$, see \cite[Theorem 1.3]{GalLuo}; thus the bound $R$ in \eqref{eq:f-preserve} can be simply taken as $\|f_0\|_p$.
The desired results follow from Lemma \ref{lem:max} and  Corollary \ref{cor-stoch-convol-1}.

We conclude this section by remarking that, by Corollary \ref{cor-stoch-convol-2}, we can also obtain an estimate on the homogeneous norm $\|f_t-\bar{f}_t\|_{\dot{H}^{-\alpha}}$ for suitable $\alpha\leq \frac d2$.

	\section{Existence of weak solutions to stochastic $2$D Euler equations}\label{subsec-existence}

The purpose of this part is to prove weak existence for stochastic 2D Euler equations \eqref{eq:SEE-S} with transport noise. First we prove the following lemma in the smooth case. Let $\mathcal S$ be the space of Schwartz test functions and $\curl= \nabla^\perp\, \cdot$ the curl operator. As in \cite[Proposition 4.1]{GalLuo}, we can also add a suitable forcing term on the right-hand side of \eqref{eq:nonlinear-regularized-spde}, but we do not pursue such generality here.

	\begin{lemma}\label{lem:apriori-estimates-nonlinear}
		Let $\omega_0\in \mathcal S$,  $W$ be a smooth noise as in Section \ref{subsec-noise}, and $\mathcal{R}$ a Fourier multiplier associated to some $r\in\mathcal S$.
		Consider the SPDE
		\begin{equation}\label{eq:nonlinear-regularized-spde}
			\d \omega+ (\mathcal{R}\, \curl^{-1}\omega)\cdot \nabla\omega\,\d t + \circ\d W \cdot \nabla\omega = 0,
			\quad \omega\vert_{t=0}=\omega_0.
		\end{equation}
		Then there exists a unique global regular strong solution to \eqref{eq:nonlinear-regularized-spde} and for any $p\in [1,\infty]$,
		\begin{equation}\label{eq:apriori-vorticity-Lp}
			\| \omega_t\|_{p}\leq \| \omega_0\|_{p} \quad \bP \mbox{-a.s. } \forall\, t\geq 0.
		\end{equation}
	\end{lemma}
	
	\begin{proof}
		We follow the arguments of \cite[Proposition 4.1]{GalLuo} and first discuss strong existence and uniqueness of regular solutions for the system \eqref{eq:nonlinear-regularized-spde}.
		Recalling that $\curl^{-1}\omega=K\ast\omega$ for the Biot--Savart kernel $K$, it holds $\mathcal R\, \curl^{-1} \omega= \widetilde K\ast \omega$ for $\widetilde K =\cF^{-1}(r)* K$, which is now a kernel in $C_b^\infty$ thanks to the assumption $r\in \mathcal S$ and properties of Biot--Savart kernel.
		
		The strong well-posedness follows from the results of \cite[Theorem 13]{CogFla}, which also yield the representation
		\begin{equation}\label{eq:proof-flow-representation}
			\omega_t (X^\omega_t(x)) = \omega_0(x)
		\end{equation}
		where $X^\omega$ is the flow associated to the Stratonovich SDE
		\begin{equation*}
			\d X^\omega_t(x)
			= b^\omega( X^\omega_t(x))\, \d t + W(\circ \d t, X_t^{\omega}(x)),
			\quad b^\omega:= \tilde K\ast \omega.
		\end{equation*}
		By the above observations $\widetilde K\ast\omega\in C^\infty_b$, which together with the results from \cite{Kunita} implies smoothness of the flow $X^\omega$; jointly with the regularity of $\omega_0$, and the representation \eqref{eq:proof-flow-representation}, this implies the desired regularity of $\omega$. Estimate \eqref{eq:apriori-vorticity-Lp} can be proved by taking the $L^p_x$-norm of both sides of \eqref{eq:proof-flow-representation} and using the measure-preserving property of $X^\omega_t: \mR^2 \to \mR^2$.
	\end{proof}
	
	Now the existence of weak solutions to \eqref{eq:SEE-I} follows from the next result by taking $r\equiv 1$.

	\begin{proposition}\label{prop:weak-existence-nonlinear}
		Let $p\in (4/3,2)$, $\omega_0\in L^p_x$ be a deterministic function;
		let $Q$ be a covariance function satisfying Assumption \ref{ass:covariance-basic} and $\mathcal{R}$ be a Fourier multiplier associated to $r\in L^\infty_x$.
		Then there exists a probabilistically weak solution $(\omega,W)$ to the SPDE
		\begin{equation}\label{eq:nonlinear-general-spde}
			\d \omega+ (\mathcal{R}\, \curl^{-1} \omega)\cdot \nabla\omega\,\d t + \d W \cdot \nabla\omega
			= \kappa \Delta \omega \, \d t,
		\end{equation}
		where the noise $W$ has covariance function $Q$; moreover $\omega$ satisfies the pathwise estimates: $\bP$-a.s. $\|\omega_t \|_{p} \leq \|\omega_0 \|_{p}$ for all $t\geq 0$.
	\end{proposition}
	
	\begin{proof}
		In light of the a priori estimates from Lemma \ref{lem:apriori-estimates-nonlinear}, we follow a compactness argument in the style of Flandoli--Gatarek \cite{FlaGat}.
		As before, we fix a finite interval $[0,T]$.
		
		First, we consider smooth approximations of \eqref{eq:nonlinear-general-spde}:
		$$\d \omega^N+ (\mathcal{R}^N\, \curl^{-1} \omega^N)\cdot \nabla\omega^N\,\d t + \d W^N \cdot \nabla\omega^N
		= \kappa \Delta \omega^N \, \d t,$$
		where the noises $W^N$ are defined by $W^N = h_{N^{-1}} \ast W$, where $h_t$ denotes the standard heat kernel;
		in particular, the associated covariances $Q^N$ are given by formula \eqref{eq:covariance-basic} with $g$ replaced by $g^N(\xi):= g(\xi) e^{-|\xi|^2/N}$, so that $W^N$ are smooth.
		Moreover, $\omega_0^N$ are smooth approximations of $\omega_0$ such that
		\begin{equation}\label{eq:proof-properties-approx1}
			\| \omega^N_0-\omega_0\|_{p} \to 0 \quad (N\to \infty)
		\end{equation}
		and satisfying the uniform bound
		\begin{equation}\label{eq:proof-properties-approx2}
			\| \omega^N_0\|_{p} \leq \| \omega_0\|_{p};
		\end{equation}
		one can verify that such approximations always exist. We also consider a sequence $r^N\in C^\infty_c$ such that $|r^N|\leq |r|$ for all $N$, $r^N\to r$ Lebesgue almost everywhere, and let $\mathcal R^N$ denote the associated Fourier multipliers.
		Set $K^N := \cF^{-1}(r^N)* K\in C^\infty_b$ and let $\omega^N$ be the solutions to the mollified It\^o SPDEs
		\begin{equation}\label{eq:nonlinear-approx-spde}
			\d \omega^N+ (K^N\ast\omega^N)\cdot \nabla\omega^N\,\d t + \d W^N \cdot \nabla\omega^N
			= \kappa_N\Delta \omega^N \, \d t
		\end{equation}
		with smooth initial data $\omega_0^N$, where $\kappa_N$ is given by the relation $Q^N(0)=2\kappa_N I_2$.
		By Lemma \ref{lem:apriori-estimates-nonlinear}, the solutions $\omega^N$ exist strongly, in particular we can take them all defined on the same probability space;
		moreover they satisfy the estimate \eqref{eq:apriori-vorticity-Lp}, where we can take the right-hand side to be independent of $N$, thanks to our choice of the approximations $(\omega_0^N, W^N, r^N)$.
		The rest of the proof is divided in two steps. \smallskip
		
		\textit{Step 1: Tightness.}
		Our first task is to derive equicontinuity estimates and moment bounds for $\{\omega^N\}_N$ in $C^0_t H^{-\beta}_x$, for suitable $\beta>0$;
		combined with \cite[Corollary A.5 and Remark A.6]{GalLuo} and the uniform pathwise bound for $\| \omega^N_t\|_{p}$, this will imply that their laws are tight in $C^0_t H^{-s}_{loc}$ for $s>2/p-1$.
		
		To this end, we need to estimate the time integrals of terms appearing in \eqref{eq:nonlinear-approx-spde} one by one.
		We start with the nonlinear term; to this end, set $u^N=\curl^{-1}\omega^N$, so that $K^N\ast \omega^N= \mathcal R^N u^N$ which
		is divergence free, we have
		\begin{align*}
			\big\| (K^N\ast \omega^N_t)\cdot\nabla \omega^N_t \big\|_{H^{-2}_x}
			& = \big\| \nabla \cdot ( (\mathcal R^N u^N_t)\, \omega^N_t) \big\|_{H^{-2}_x}\\
			& \lesssim \big\| (\mathcal R^N u^N_t)\, \omega^N_t \big\|_{H^{-1}_x}
			\lesssim \big\| (\mathcal R^N u^N_t)\, \omega^N_t \big\|_{\tilde p},
		\end{align*}
		where $\tilde p\in (1,p)$ will be chosen later. By H\"older's inequality with exponents $\frac1{\tilde p}= \frac1p +\frac1q$,
		$$\aligned
		\big\| (K^N\ast \omega^N_t)\cdot\nabla \omega^N_t \big\|_{H^{-2}_x} &\lesssim \| \mathcal R^N u^N_t\|_{q} \| \omega^N_t\|_{p} \lesssim \| \mathcal R^N u^N_t\|_{\dot H^a_x} \| \omega^N_t\|_{p},
		\endaligned $$
		where the parameter $a$ satisfies $\frac1q= \frac12 -\frac{a}2$. Since $\mathcal R^N$ is defined by a Fourier multiplier which is bounded uniformly in $N$, we have
		$$\big\| (K^N\ast \omega^N_t)\cdot\nabla \omega^N_t \big\|_{H^{-2}_x} \lesssim \| u^N_t\|_{\dot H^{a}_x} \| \omega^N_t\|_{p} \leq \| \omega^N_t\|_{\dot H^{a-1}_x} \| \omega^N_t\|_{p}. $$
		We want that $\| \omega^N_t\|_{\dot H^{a-1}_x} \lesssim \| \omega^N_t\|_{p}$ which, by Sobolev embedding in 2D, results in $\frac1p= 1-\frac{a}2$; therefore, choosing $\tilde p= 2p/(4-p) >1$ which requires $p>4/3$, and $q= 2p/(2-p)$, $a= 2-2/p$, the above computations hold true.  To sum up, we obtain
		$$\big\| (K^N\ast \omega^N_t)\cdot\nabla \omega^N_t \big\|_{H^{-2}_x} \lesssim \| \omega^N_t\|_{p}^2 \leq \| \omega_0\|_{p}^2,$$
where we have used the pathwise uniform $L^p$-bound $\{ \omega^N_t\}_{N\geq 1, t\in [0,T]}$. As a consequence,
		\begin{equation}\label{eq:equicontinuity-1}
			\sup_N\, \bE\bigg[\, \Big\| \int_0^\cdot \big[(K^N\ast \omega^N_t)\cdot\nabla \omega^N_t \big]\, \d t \Big\|_{C^{1/2}_t H^{-2}_x} \bigg]
			\lesssim \bigg[\int_0^T \| \omega_0\|_{p}^4\, \d t \bigg]^{1/2}
			< \infty.
		\end{equation}

		For the stochastic integrals, since $W^N_t= \sum_k \sigma_k^N W^k_t$ are divergence free in the space variable, for any $q\ge 1$, it holds
  \begin{equation}\label{proof-martingale}
  \aligned
  \bE\bigg[ \Big\| \int_s^t \nabla\omega^N_r \cdot \d W^N_r \Big\|_{H^{-2}_x}^{2q} \bigg]
  & = \bE \bigg[ \Big\| \nabla\cdot \int_s^t \omega^N_r\, \d W^N_r \Big\|_{H^{-2}_x}^{2q} \bigg] \\
  & \lesssim \bE \bigg[ \Big\| \int_s^t \omega^N_r\, \d W^N_r \Big\|_{H^{-1}_x}^{2q} \bigg]\\
  & \lesssim \bE\bigg[\Big( \sum_k \int_s^t \| \omega^N_r \sigma_k^N \|_{H^{-1}_x}^2\, \d r \Big)^q \bigg] .
  \endaligned
  \end{equation}
We have
  $$\sum_k \| \omega^N_r \sigma_k^N \|_{H^{-1}_x}^2 = \sum_k \int \<\xi\>^{-2} \big|\mathcal F(\omega^N_r \sigma^N_k)(\xi) \big|^2\,\d \xi = \sum_k \int \<\xi\>^{-2} \big|\big(\widehat\omega^N_r\ast\widehat \sigma^N_k \big)(\xi) \big|^2\,\d \xi $$
and, similarly to the treatment of \eqref{proof-stoch-convol-2},
  $$\aligned
  \sum_k \big|\big(\widehat\omega^N_r\ast\widehat \sigma^N_k \big)(\xi) \big|^2
  &= \sum_k \iint \widehat\omega^N_r(\xi-\eta) \, \overline{\widehat\omega^N_r(\xi-\zeta)}\, \widehat \sigma^N_k(\eta) \cdot \overline{\widehat \sigma^N_k(\zeta)}\,\d\eta \d\zeta \\
  &= \int |\widehat\omega^N_r(\xi-\eta)|^2 \, {\rm Tr}(\widehat{Q}^N (\eta)) \,\d\eta \\
  &= \big(|\widehat\omega^N_r|^2 \ast {\rm Tr}(\widehat{Q}^N)\big)(\xi);
  \endaligned $$
therefore, for any $a>1$, H\"older's inequality implies
  $$\aligned
  \sum_k \| \omega^N_r \sigma_k^N \|_{H^{-1}_x}^2
  &= \int \<\xi\>^{-2} \big(|\widehat\omega^N_r|^2 \ast {\rm Tr}(\widehat{Q}^N) \big)(\xi) \,\d\xi \\
  &\leq \big\|\<\cdot \>^{-2} \big\|_a \big\||\widehat\omega^N_r|^2 \ast {\rm Tr}(\widehat{Q}^N) \big\|_{a'} \\
  &\lesssim_a \big\| \widehat\omega^N_r\big\|_{2a'}^2 \|{\rm Tr}(\widehat{Q}^N) \|_{1},
  \endaligned $$
where in the last step we have used Young's inequality. As in the proof of Lemma \ref{lem:max-2}, we take $2a'$ to be the conjugate number of $p$, that is, $a= p/(2-p)>1$, then we arrive at
  $$\sum_k \| \omega^N_r \sigma_k^N \|_{H^{-1}_x}^2 \lesssim \big\| \omega^N_r \big\|_{p}^2\, \| \widehat{Q}^N \|_{1} \lesssim \| \omega_0 \|_{p}^2 \| \widehat{Q} \|_{1}, $$
where the last step is due to the uniform $L^p_x$-bound on $\{\omega^N_r \}_N$ and the fact that $\sup_N \| \widehat{Q}^N \|_{1} \lesssim \| \widehat{Q} \|_{1}<+\infty$. Inserting this estimate into
\eqref{proof-martingale} gives us
  \begin{align*}
			\bE\bigg[ \Big\| \int_s^t \nabla\omega^N_r \cdot \d W^N_r \Big\|_{H^{-2}_x}^{2q} \bigg]
			\lesssim \|\omega_0\|_{p}^{2q} (t-s)^q.
		\end{align*}
Combining the above estimate with Kolmogorov's continuity theorem, one can conclude that for any $\gamma<1/2$ and any $q\in [1,\infty)$ it holds
		\begin{equation}\label{eq:equicontinuity-2}
			\sup_N \, \bE\bigg[ \Big\| \int_0^\cdot \nabla\omega^N_r\cdot \d W^N_r\Big\|_{C^\gamma_t H^{-2}_x}^q \bigg] <\infty.
		\end{equation}
		Next, again by the uniform pathwise bounds on $\| \omega^N\|_{L^p_x}$, it holds
		\begin{equation}\label{eq:equicontinuity-3}
			\aligned
			\sup_N\, \bE\bigg[ \Big\| \int_0^\cdot \kappa_N \Delta \omega^N_r\, \d r \Big\|_{C^{1/2}_t H^{-3}_x} \bigg]
			&\lesssim \sup_N\, \bE \bigg[ \int_0^T \| \omega^N_r\|_{H^{-1}_x}^2\, \d r \bigg]^{1/2} \\
			&\leq \sup_N\, \bE \bigg[ \int_0^T \| \omega^N_r\|_{p}^2 \, \d r \bigg]^{1/2} \\
            &\leq T^{1/2} \|\omega_0 \|_p < \infty.
			\endaligned
		\end{equation}

		Writing the SPDEs \eqref{eq:nonlinear-approx-spde} in integral form and combining the estimates \eqref{eq:equicontinuity-1}, \eqref{eq:equicontinuity-2} and \eqref{eq:equicontinuity-3}, we can conclude that the family $\{\omega^N\}_N$ is equicontinuous in $C^0_t H^{-3}_x$, with suitable moment bounds. Moreover, since $L^p_x \hookrightarrow H^{1-2/p}_x$, thus we have the pathwise bound
  $$\sup_N \| \omega^N\|_{L^\infty_t H^{1-2/p}_x} \lesssim \sup_N \| \omega^N\|_{L^\infty_t L^p_x} \leq \| \omega_0 \|_{p}; $$
		by \cite[Remark A.6]{GalLuo} this implies tightness of the laws of the sequence in $C^0_t H^{-s}_{loc}$, for any $s>2/p-1$. Note that $ 2/p-1 <1 $ since $p >1$.

		\textit{Step 2: Passage to the limit in a new probability space.}
		The next step is very classical, therefore we mostly sketch it; for a thorough presentation in a similar setting, see for instance \cite[Proof of Theorem 2.2]{FGL21}.
		
		By construction $W^N\to W$ in $L^2_\Theta C^0_t L^2_{loc}$; together with Step 1, this implies tightness of the laws of $\{(\omega^N,W^N)\}_N$ in $C^0_t H^{-s}_{loc}\times C^0_t L^2_{loc}$ for any $s\in (2/p-1,1)$.
		By Prohorov's theorem and Skorohod's representation theorem, we can extract a (not relabelled) subsequence and construct a new probability space $(\tilde\Theta, \tilde{\mathbb F}, \tilde\bP)$, on which there exists a sequence of processes $\{(\tilde \omega^N,\tilde W^N)\}_N$ satisfying:
		\begin{itemize}
			\item[i)] for any $N\geq 1$, $(\tilde \omega^N, \tilde W^N)$ has the same law as  $(\omega^N, W^N)$;
			\item[ii)] $\tilde\bP$-a.s., $(\tilde \omega^N, \tilde W^N)\to (\tilde \omega,\tilde W)$ in $C^0_t H^{-s}_{loc}\times C^0_t L^2_{loc}$.
		\end{itemize}
	  Item i) above implies that $\tilde\omega^N$ is a solution to \eqref{eq:nonlinear-approx-spde}, adapted to the filtration generated by $\tilde W^N$, satisfying the same pathwise bound \eqref{eq:apriori-vorticity-Lp}; moreover, $\tilde W^N$ are again $Q^N$-Brownian motions.
		By Item ii) and uniform $L^p_x$-bounds on $\{\tilde\omega^N \}_N$, it then follows that, on a set of full probability $\tilde\bP$, $\tilde\omega^N_t\rightharpoonup \tilde \omega_t$ in $L^p_x$ for all $t\in [0,T]$; moreover, since $\tilde W^N\to \tilde W$, $\tilde W$ is a $Q$-Brownian motion.
		With slightly more effort, using i) and ii) above, one can additionally show that $\tilde W$ is a $\tilde{\mathbb F}_t$-Brownian motion, where $\tilde {\mathbb F}_t=\sigma\big(\tilde W_r, \tilde \omega_r: r\leq t \big)$. 
  
		We claim that $\big( \tilde\Theta, \tilde{\mathbb F}, \tilde {\mathbb F}_t, \tilde \bP; \tilde \omega, \tilde W \big)$ is a desired weak solution to \eqref{eq:nonlinear-general-spde}. This can be accomplished by applying Items i) and ii) and passing to the limit as $N\to\infty$ in the weak form of the SPDEs \eqref{eq:nonlinear-approx-spde} satisfied by $\tilde\omega^N$.
		For simplicity, let us drop the tildes in the notation of processes, so that for any $\varphi\in C_c^\infty$, $\tilde\bP$-a.s. for all $t\in [0,T]$, it holds
		\begin{equation}\label{eq:nonlinear-approx-spde-new}\begin{split}
				\<\omega^N_t,\varphi\>
				& = \<\omega^N_0,\varphi\> + \int_0^t \big\<\omega^N_s, (\mathcal R^N u^N_s) \cdot \nabla \varphi \big\>\,\d s + \int_0^t \big\<\omega^N_s, \d W^N_s\cdot\nabla \varphi \big\> \\
				& \quad + \int_0^t \kappa_N \< \omega^N_s, \Delta \varphi \>\,\d s.
		\end{split}\end{equation}
The nonlinear terms $\<\omega^N_s, (\mathcal R^N u^N_s) \cdot \nabla \varphi\>$ are the only ones requiring less standard treatment. We have
  $$\aligned
  &\big\<\omega^N_s, (\mathcal R^N u^N_s) \cdot \nabla \varphi \big\> - \big\<\omega_s, (\mathcal R u_s) \cdot \nabla \varphi \big\> \\
  &= \big\<\omega^N_s- \omega_s, (\mathcal R^N u^N_s) \cdot \nabla \varphi \big\> + \big\<\omega_s, (\mathcal R^N u^N_s- \mathcal R u_s) \cdot \nabla \varphi \big\>,
  \endaligned $$
where the two quantities will be denoted by $I_1$ and $I_2$, respectively. For a compact set $K\subset \mathbb R^2$ and $\gamma\geq 0$, we write $H^\gamma_K$ for those elements of $H^\gamma_x$ which are supported in $K$. Due to the presence of test function $\varphi$, we can regard $\mathcal R^N u^N_s$ as a distribution restricted to ${\rm supp}(\varphi)$; then,
  $$\aligned
  |I_1| &\leq \big\|(\omega^N_s- \omega_s)\nabla\varphi \big\|_{H^{-\gamma}_{\rm supp(\varphi)}} \big\|\mathcal R^N u^N_s \big\|_{H^\gamma_{\rm supp(\varphi)}}\\ 
  & \leq C_\varphi \big\|(\omega^N_s- \omega_s)\nabla\varphi \big\|_{H^{-\gamma}_{x}} \big\|\mathcal R^N u^N_s \big\|_{\dot H^\gamma_{\rm supp(\varphi)}} ,
  \endaligned $$
where the last step is due to the fact that compactly supported nonhomogeneous norms are equivalent to the corresponding homogeneous ones, cf. \cite[Proposition 1.55]{BCD}. One has
  $$\big\|\mathcal R^N u^N_s \big\|_{\dot H^\gamma_{\rm supp(\varphi)}} \leq \big\|\mathcal R^N u^N_s \big\|_{\dot H^\gamma_x} \lesssim \big\| u^N_s \big\|_{\dot H^\gamma_x} \lesssim \big\| \omega^N_s \big\|_{\dot H^{\gamma-1}_x} \lesssim \big\| \omega^N_s \big\|_{L^p_x} \leq \| \omega_0 \|_{L^p_x} $$
if we take $\gamma=2-2/p$, which is greater than $2/p-1$ since $p>4/3$. By Item ii) above, we know that, for any $\gamma>2/p-1$, $\tilde{\bP}$-a.s., $\|\omega^N_s- \omega_s\|_{C_t H^{-\gamma}_{loc}} \to 0$ as $N\to \infty$; therefore,
  $$\lim_{N\to \infty} I_1=0. $$

Now we turn to show that $I_2 $ vanishes as $N\to \infty$; it holds
  $$I_2= \big\<\omega_s, (\mathcal R^N- \mathcal R) u^N_s \cdot \nabla \varphi \big\> + \big\<\omega_s, \mathcal R (u^N_s- u_s) \cdot \nabla \varphi \big\>. $$
Thanks to the bound on $\omega$ and uniform estimates on $u^N$, the first term can be easily shown to vanish by the assumption on $\mathcal R^N$ and dominated convergence theorem. Thus we focus on the second term. Item ii) implies that $\omega^N$ converges $\tilde{\bP}$-a.s. to $\omega$ in $C_t H^{-\gamma}_{loc}$ for any $\gamma>2/p-1$. Combining with the boundedness of $\{\omega^N\}_N$ in $\dot H^{1-2/p}_x$, we deduce that, for any $s\in [0,T]$, $\omega^N_s$ converges weakly in $\dot H^{1-2/p}_x$ to $\omega_s$; therefore, $ \mathcal R u^N_s=  \mathcal R {\rm curl}^{-1} \omega^N_s$ converges weakly in $\dot H^{2-2/p}_x$ to $ \mathcal R u_s=  \mathcal R {\rm curl}^{-1} \omega_s$. Applying again \cite[Proposition 1.55]{BCD}, we see that, as $N\to \infty$,
  $$\mathcal R u^N_s \cdot \nabla\varphi \rightharpoonup \mathcal R u_s \cdot \nabla\varphi \quad \mbox{in } H^{2-2/p}_x. $$
As $H^{2-2/p}_x \subset \dot H^{2/p-1}_x$ and $\omega_s \in \dot H^{1-2/p}_x$, one has, $\tilde{\bP}$-a.s. for all $s\in [0,T]$,
  $$\lim_{N\to \infty}  \big\<\omega_s, \mathcal R (u^N_s- u_s) \cdot \nabla \varphi \big\> =0. $$
Summing up the above arguments, we conclude the convergence of nonlinear terms.
	\end{proof}

	\section{Quantitative estimates}

 This section consists of two parts: we first present a well-posedness result for the vorticity form of deterministic 2D Navier-Stokes equation with $L^p_x$ initial data, then we provide the proof of Theorem \ref{thm:SEE} in the second part. 

 \subsection{Well-posedness of deterministic 2D Navier-Stokes equation}

	The following result on the existence and uniqueness of solutions to deterministic 2D Navier-Stokes equation in vorticity form does not seem to have been explicitly stated in the literature, thus we present here its proof for completeness.
	
	\begin{proposition}\label{prop-2D-NSE}
		Let $p\in (1,2)$. Assume that $\omega_0\in L^p(\mR^2)$, then \eqref{eq:NS} admits a unique solution in $C([0,T];L^p)$. Moreover, such a solution satisfies
		\begin{equation}\label{eq:bar-w}
			\|{\omega}_t \|_{p} \leq \|\omega_0\|_p ~\mbox{ and }~ \|\nabla {\omega}_t \|_{p} \leq C t^{-\frac{1}{2}}\|\omega_0\|_p,
		\end{equation}
		where $C$ only depends on $p$ and $T$. Consequently, the solution also belongs to $L^{\frac{2}{\alpha}} ([0,T]; H^{\alpha, p})$ for all $\alpha\in [0,1)$; here $H^{\alpha, p}$ is the Bessel potential space.
	\end{proposition}
	\begin{proof}
        For notational simplicity, we will always take $\kappa=1$ and $h_t(x)=(4\pi t)^{-1} \e^{-|x|^2/4t}$. 
        
        {\it Existence.}
        Given $\omega_0\in L^p$, set 
        \[
          \cX= \left\{\omega \in C([0,T]; L^p): \omega|_{t=0}=\omega_0 ~\mbox{ and }~ \|\omega\|_{C_tL^p_x}\leq \|\omega_0\|_p\right\}. 
        \]
        Take $\omega\in \cX$ and set $u=K*\omega$, where $K$ is the Biot-Savart kernel. Since 
        \[
            \|u\|_{L^\infty_t L^{2p/(2-p)}_x} \lesssim_{p}\|\omega\|_{C_tL^p_x}\lesssim_p \|\omega_0\|_p,   
        \]
        it follows from \cite[Theorem 1.1]{CHXZ17} and the condition $\div u=0$ that $\p_t-\Delta+u\cdot\nabla$ admits a fundamental solution, which we denote by $p^{\omega}_{s,t}(x,y)$, satisfying the following properties: 
        \begin{equation}\label{eq:Conservative}
            \int_{\mR^2} p^{\omega}_{s,t}(x,y)\, \d y = \int_{\mR^2}  p^{\omega}_{s,t}(x,y)\, \d x =1, 
        \end{equation}
        \begin{equation}\label{eq:HKE}
            p^{\omega}_{s,t}(x,y) \lesssim h_{\lambda(t-s)}(x-y) 
        \end{equation}  
        and 
        \begin{equation}\label{eq:Dheat}
            \nabla_y p^{\omega}_{s,t}(x,y)\lesssim t^{-\frac{1}{2}} h_{\lambda(t-s)}(x-y), 
        \end{equation}
        where $\lambda>0$ only depends on $p$ and $\|\omega_0\|_p$. 
        Consider the map 
        \begin{equation*}
            \cX \ni \omega \mapsto \Phi( \omega)_t(y) = \int_{\mR^2} p^{\omega}_{0, t}(x,y) \omega_0 (x) \,\d x, 
        \end{equation*}
        i.e., $\Phi(\omega)$ is the solution of equation 
        \[
            \p_t \xi + u\cdot \nabla \xi = \Delta \xi , \quad \xi|_{t=0}=\omega_0. 
        \]
        Since $\div u=0$, integration by parts yields $\|\Phi(\omega)_t\|_p=\|\xi_t\|_p\leq \|\omega_0\|_p$. We further claim that \(\Phi\) is a continuous map on $\cX$ and that \(\Phi(\cX) \subseteq \cX\) is relatively compact in \(C([0,T]; L^p)\). If it holds, then  by Schauder's fixed point theorem, $\Phi$ has a fixed point in $\cX$, which is a solution to \eqref{eq:NS}. 
        
        Therefore, our task becomes to show the above claim. We start with the equation: 
        \[
            p^{\omega}_{s,t}(x,y)= h_{t-s}(x-y)+\int_s^t \d \tau \int_{\mR^2} p^\omega_{s,\tau}(x,z)\,u_\tau (z)\cdot \nabla h_{t-\tau} (z-y) \,\d z, 
        \]
        (see for instance \cite{CHXZ17} or \cite{ZH24}). From this,  we have 
        \[
            \Phi(\omega)_t= P_{t} \omega_0+\int_{\mR^2} \omega_0(x )\, \d x\int_0^t \d \tau \int_{\mR^2} p^\omega_{0,\tau}(x,z)\,u_\tau (z)\cdot \nabla h_{t-\tau} (z-\cdot)\, \d z. 
        \]
        In the light of \cite[Lemma 2.4]{ZH24}, we find 
        \begin{align*}
            & \left\| \int_{\mR^2} \omega_0(x )\, \d x\int_0^t \d \tau \int_{\mR^2} p^\omega_{0,\tau}(x,z)\,u_\tau (z)\cdot \nabla h_{t-\tau} (z-y)\, \d z \right\|_{L^p_y} \\
            & \lesssim  t^{1-\frac{1}{p}} \left\| \int_{\mR^2} |\omega_0(x )| h_{\lambda t}(x-y) \, \d x \right\|_{L^p_y} \lesssim t^{1-\frac{1}{p}} \|\omega_0\|_p. 
        \end{align*}
        This leads to 
        \begin{align}\label{eq:UC}
            \|\Phi(\omega)_{t}-\omega_0\|_p & \lesssim \|(P_t-I)\omega_0\|_p+t^{1-\frac{1}{p}} \|\omega_0\|_p\to 0 \quad (t\to 0).  
        \end{align}
        Next, according to \cite[Lemma 3.8]{ZH24}, for some $\beta\in (0,1)$ and any $0<s\leq t\leq T$, \begin{align*}
            \|\Phi(\omega)_t-\Phi(\omega)_s\|_p &\lesssim \left\| \int_{\mR^2}|\omega_0(x)|\, |p^\omega_{0,s}(x,y)-p^\omega_{0,t}(x,y)|\, \d x \right\|_{L^p_y} \lesssim s^{-\frac{\beta}{2}} |s-t|^{\beta} \|\omega_0\|_p. 
        \end{align*}
        Together with \eqref{eq:UC}, this implies that \(\Phi(\mathcal{X})\) is equicontinuous in \(C([0,T]; L^p)\). Additionally, from \eqref{eq:Dheat}, we have 
        \[
            \|\nabla \Phi(\omega)_t\|_p \lesssim t^{-\frac{1}{2}} \|\omega_0\|_p, 
        \]
        indicating that for each \(t \in [0,T]\), \(\Phi(\mathcal{X})_t\) is compact in \(L^p\). By the Arzelà–Ascoli theorem, we can conclude that \(\Phi(\mathcal{X})\) is relatively compact in \(C([0,T]; L^p)\).

        The construction above shows that each fixed point of \(\Phi\) belongs to \(L^\infty_t L^p_x \cap L^{2,\infty}_t H^{1,p}_x\). The interpolation theorem implies that these fixed points are also contained in \(L^{\frac{2}{\alpha}}([0,T]; H^{\alpha,p})\) for any $\alpha\in [0,1)$. 
        
        {\it Uniqueness.} 
		Let $\omega^1$ and $\omega^2$ be two solutions of \eqref{eq:NS} in $C([0,T]; L^p)$ with the same initial data $\omega_0$. Set $u^i=K*\omega^i$,  $u=u^1-u^2$ and $\omega=\omega^1-\omega^2$. By the discussion above, we have 
	    \begin{equation*}
	      \omega^i_t (y) = \int_{\mR^2} p_{0, t}^{i}(x,y) \omega_0 (x)\, \d x,  
	    \end{equation*}
        where $p^{i}_{s,t}(x,y)$ is the fundamental solution of  $\partial_t- \Delta+u^i\cdot\nabla$. Together with \eqref{eq:HKE} and Young's inequality \eqref{Eq:Young1}, this yields \begin{equation}\label{eq:decay}
			\sup_{t\in [0,T]} t^{\frac{1}{p}-\frac{1}{r}} \|\omega^i_t\|_{L^{r,1}} \lesssim \sup_{t\in [0,T]} t^{\frac{1}{p}-\frac{1}{r}} \|h_{\lambda t}* \omega_0\|_{L^{r,1}} \lesssim \|\omega_0\|_p<\infty,  \quad r>p. 
		\end{equation}
        Noting that 
		\begin{equation*}
			\begin{aligned}
				\omega^i_t= P_t \omega_0 -\int_0^t \nabla\cdot P_{t-s} ( u^i_s\omega^i_s) \, \d s, 
			\end{aligned}
		\end{equation*}
		we have 
		\begin{equation}\label{eq:winteg}
			\begin{aligned}
				\omega_t &=\omega^1_t-\omega^2_t = -\underbrace{\int_0^t \nabla\cdot P_{t-s}(u^1_s\,\omega_s)\, \d s}_{I_t} - \underbrace{\int_0^t \nabla\cdot P_{t-s}(u_s\,\omega^2_s)\, \d s}_{J_t}. 
			\end{aligned}
		\end{equation}
		For any $r\in (p, \infty]$, following \cite{GigMiyOsa88}, we set 
		\[
		    \|f\|_{p,T}:= \sup_{t\in [0,T]} t^{\frac{1}{p}-\frac{1}{r}} \|f(t)\|_r. 
		\]
		By \eqref{Eq:Young2} and \eqref{eq:decay}, we see that for any $s\in [0,T]$, 
		\begin{align*}
			\|u^1_s\, \omega_s\|_{r}\leq\|u^1_s\|_\infty \|\omega_s\|_r \leq C \|\omega^1_s\|_{L^{2,1}} \|\omega_s\|_r \lesssim  s^{\frac{1}{2}-\frac{1}{p}} \|\omega_s\|_r. 
		\end{align*}
		This yields that for any $r\in (p,2)$, 
		\begin{equation}\label{eq:vorticityI}
			\begin{aligned}
				\|I\|_{p,T} & \leq\sup_{t\in [0,T]} t^{\frac{1}{p}-\frac{1}{r}} \l[\int_0^t (t-s)^{-\frac{1}{2}} s^{\frac{1}{2}+\frac{1}{r}-\frac{2}{p}}\|\omega\|_{p,T}\, \d s\r] 
				\lesssim T^{1-\frac{1}{p}} \|\omega\|_{p,T}, 
			\end{aligned}
		\end{equation}
        
		Regarding $J_t$, by Young's inequality again and \eqref{eq:decay}, 
		\begin{align*}
			\|u_s\omega^2_s\|_{r} \leq \|u_s\|_{l} \|\omega^2_s\|_{2} \lesssim \|\omega_s\|_{r} \, s^{\frac{1}{2}-\frac{1}{p}} 
		\end{align*}
		with $r\in (1, 2)$ and $\frac{1}{l}=\frac{1}{r}-\frac{1}{2}$. Thus, 
		\begin{equation}\label{eq:vorticityJ}
			\begin{aligned}
				\|J\|_{p,T} & \leq \sup_{t\in [0,T]}  t^{\frac{1}{p}-\frac{1}{r}}  \l[ \int_0^t (t-s)^{-\frac{1}{2}} s^{\frac{1}{2}+\frac{1}{r}-\frac{2}{p}}\|\omega\|_{p,T}\, \d s\r] \lesssim T^{1-\frac{1}{p}} \|\omega\|_{p,T}. 
			\end{aligned}
		\end{equation}
		Combining \eqref{eq:winteg}, \eqref{eq:vorticityI} and \eqref{eq:vorticityJ}, we get 
		\[
		\|\omega\|_{p,T}\leq C'T^{1-\frac{1}{p}} \|\omega\|_{p,T}, 
		\]
        where $C'$ only depends on $p, r$ and $\|\omega_0\|_p$. Choosing $T_0>0$ sufficiently small such that $C'T_0^{1-\frac{1}{p}}<1$, we obtain the desired result for $t\in [0,T_0]$. Repetition of the above process ensures the desired uniqueness result. 
	\end{proof}
	
	\subsection{Proof for Theorem \ref{thm:SEE}}
	
	\begin{proof}[Proof of Theorem \ref{thm:SEE}]
		Let $\bar{\omega} \in C([0,T]; L^p)$ 
		be the unique solution to \eqref{eq:NS}.
		Rewriting \eqref{eq:SEE-I} and \eqref{eq:NS} in mild form, and taking difference, we obtain
		$$
		\omega_t -\bar\omega_t= -\int_0^t P_{t-s} (u_s\cdot\nabla\omega_s- \bar u_s\cdot\nabla \bar\omega_s)\,\d s + Z_t,
		$$
		where $Z_t$ is the stochastic convolution
		$$
		Z_t= - \int_0^t P_{t-s} \nabla \cdot(\omega_s\, \d W_s).
		$$
		We rewrite the above equation as
		\begin{equation}\label{proof-difference}
			\omega_t -\bar\omega_t= -\int_0^t P_{t-s}( (u_s-\bar u_s)\cdot\nabla\omega_s)\,\d s - \int_0^t P_{t-s} (\bar u_s\cdot\nabla (\omega_s- \bar\omega_s))\,\d s + Z_t,
		\end{equation}
		and denote the first two terms on the right-hand side as $I^1_t$ and $I^2_t$.
		For $\alpha$ given as in the statement of Theorem \ref{thm:SEE}, we have
		$$
		\aligned
		\|I^1_t\|_{\dot{H}^{-\alpha}} &\leq \int_0^t \l\|P_{t-s}( (u_s-\bar u_s)\cdot\nabla\omega_s) \r\|_{\dot{H}^{-\alpha}} \,\d s \\
		& \lesssim_p \int_0^t \frac1{(t-s)^{1/p}} \|(u_s-\bar u_s)\cdot\nabla\omega_s \|_{\dot{H}^{-\alpha-2/p}}\,\d s
		\endaligned
		$$
by \eqref{eq:hke1}.	Using the divergence free property of $u$ and $\bar{u}$, we arrive at
	\begin{equation*}
		\begin{aligned}
		\|I^1_t\|_{\dot H^{-\alpha}} &\lesssim \int_0^t \frac1{(t-s)^{1/p}} \|(u_s-\bar u_s)\omega_s \|_{\dot{H}^{1-\alpha-2/p}} \,\d s \\
		&\lesssim_{\alpha, p} \int_0^t \frac1{(t-s)^{1/p}} \|(u_s-\bar u_s)\omega_s \|_{\frac{2p}{\alpha p+2}}\,\d s \\
        &\leq \int_0^t \frac1{(t-s)^{1/p}} \| u_s-\bar u_s \|_{2/\alpha} \| \omega_s \|_{p} \,\d s
		\end{aligned}
		\end{equation*}
by Sobolev embedding (which requires $\alpha<2-2/p$) and H\"older's inequality. Applying again the Sobolev embedding inequality yields
	\begin{equation*}
		\begin{aligned}
		\|I^1_t\|_{\dot H^{-\alpha}}
		&\lesssim_\alpha \int_0^t \frac1{(t-s)^{1/p}} \|u_s-\bar u_s \|_{\dot{H}^{1-\alpha}} \|\omega_s \|_{p}\,\d s \\
		&\leq \|\omega_0 \|_{p} \int_0^t \frac1{(t-s)^{1/p}} \|\omega_s-\bar{\omega}_s \|_{\dot{H}^{-\alpha}} \,\d s \\
		&\leq  \|\omega_0 \|_{p} \bigg(\int_0^t \frac{\d s}{(t-s)^{q'/p}} \bigg)^{1/q'} \bigg(\int_0^t  \|\omega_s-\bar{\omega}_s \|_{\dot{H}^{-\alpha}}^q\,\d s \bigg)^{1/q} ,
		\end{aligned}
		\end{equation*}
		where in the last step we have used H\"older's inequality with exponents $\frac1{q'}+ \frac1q=1$. Since $q>p'$, we have $q'/p<1$, thus we obtain
		\begin{equation}\label{eq:aprox1}
			\|I^1_t\|_{\dot{H}^{-\alpha}}^q \lesssim_{\alpha, p, q, T} \|\omega_0 \|_{p}^q\int_0^t  \|\omega_s-\bar{\omega}_s \|_{\dot{H}^{-\alpha}}^q\,\d s.
		\end{equation}
		
		Next we estimate $I^2_t$: by \eqref{eq:hke1},
		\begin{equation}\label{eq:aprox2}
			\begin{aligned}
			\|I^2_t\|_{\dot{H}^{-\alpha}} &\lesssim \int_0^t \|P_{t-s}(\bar u_s\cdot\nabla(\omega_s-\bar\omega_s)) \|_{\dot{H}^{-\alpha}} \,\d s \\
			&\lesssim \int_0^t \frac1{(t-s)^{1/2}} \|\bar u_s\cdot\nabla(\omega_s-\bar\omega_s) \|_{\dot{H}^{-\alpha-1}} \,\d s \\
			&\lesssim \int_0^t \frac1{(t-s)^{1/2}} \|\bar u_s (\omega_s-\bar\omega_s) \|_{\dot{H}^{-\alpha}} \,\d s.
			\end{aligned}
		\end{equation}
		Denote $v=\omega-\bar{\omega}$ and $\Lambda= (-\Delta)^{1/2}$; then
		\begin{equation}\label{eq:prod-H}
			\begin{aligned}
				\|\bar{u}\, v\|_{\dot{H}^{-\alpha}} &= \sup_{\| \varphi \|_2\leq 1} \<\bar{u}\, v, \Lambda^{-\alpha}\varphi\>=  \sup_{\| \varphi \|_2\leq 1} \<v, \bar{u}\, \Lambda^{-\alpha} \varphi\> \\
&\leq \|v\|_{\dot{H}^{-\alpha}} \sup_{\| \varphi \|_2\leq 1} \| \bar{u}\, \Lambda^{-\alpha} \varphi \|_{\dot{H}^\alpha}
=\|v\|_{\dot{H}^{-\alpha}} \sup_{\| \varphi \|_2\leq 1}  \|\Lambda^{\alpha}(\bar{u}\, \Lambda^{-\alpha} \varphi)\|_2.
			\end{aligned}
		\end{equation}
		By \eqref{eq:D-prod}, we have
		\begin{equation}\label{eq:D-prod0}
			\Lambda^{\alpha}(\bar{u}\Lambda^{-\alpha} \varphi)= \Lambda^{\alpha} \bar{u}  \, \Lambda^{-\alpha} \varphi +  \bar{u} \varphi - \Gamma_{\alpha}(\bar{u}, \Lambda^{-\alpha} \varphi),
		\end{equation}
		where $\Gamma_\alpha$ is given by \eqref{eq:Gamma}.
		
		Recall that $\alpha\in \big(\frac{2}{p}-1, 2-\frac{2}{p}\big)$; using H\"older's inequality and Sobolev embedding, we get
		\begin{equation}\label{eq:D-prod1}
			\begin{aligned}
			\| \Lambda^{\alpha} \bar{u}  \, \Lambda^{-\alpha} \varphi\|_2
&= \| \Lambda^{\alpha} (K\ast\bar{\omega})  \, \Lambda^{-\alpha} \varphi\|_2 \lesssim\| \Lambda^{\alpha-1} \bar{\omega}\|_{2/\alpha}  \| \Lambda^{-\alpha}\varphi\|_{2/(1-\alpha)}\\
			& \lesssim \|\bar{\omega}\|_{2} \|\varphi\|_2 \lesssim \|\bar{\omega}\|_{H^{\alpha, p}} \|\varphi\|_2,
			\end{aligned}
		\end{equation}
		where in the last step we used the relation $\frac{1}{p}>\frac{1}{2}>\frac{1}{p}-\frac{\alpha}{2}$.
		
		Similarly, noting that $\alpha>\frac{2}{p}-1$, one can see that
		\begin{equation}\label{eq:D-prod2}
			\begin{aligned}
			\|\bar{u}\,\varphi\|_{2} &\leq \|\bar{ u }\|_\infty \|\varphi\|_2 \lesssim \|\bar u\|_{H^{\alpha, 2p/(2-p)}} \|\varphi\|_2\\
				& =\|K\ast \bar{\omega}\|_{H^{\alpha, 2p/(2-p)}}\|\varphi\|_2 \lesssim \|\bar{\omega}\|_{H^{\alpha, p}}\|\varphi\|_2,
			\end{aligned}
		\end{equation}
		where in the last step we used \eqref{Eq:Young1}.
		
For the last term of the right-hand side of \eqref{eq:D-prod0}, we set $\delta= 1+\alpha-\frac{2}{p}>0$ and $\theta=1-\frac{2-p}{\alpha p}\in (0,1)$. Thanks to \eqref{eq:dif-Lp} and Sobolev embedding, we have
		\[
		\l\| \Lambda^{-\alpha} \varphi(\cdot+z)-\Lambda^{-\alpha}\varphi(\cdot)\r\|_2\lesssim |z|^{\alpha} \|\varphi\|_2
		\]
		and
		\[
		\l\| \Lambda^{-\alpha} \varphi(\cdot+z)-\Lambda^{-\alpha}\varphi(\cdot)\r\|_{\frac{2}{1-\alpha}}\lesssim \|\varphi\|_2,
		\]
		which together with interpolation inequalities yield
		\[
		\l\| \Lambda^{-\alpha} \varphi(\cdot+z)-\Lambda^{-\alpha}\varphi(\cdot)\r\|_{\frac{p}{p-1}}\lesssim |z|^{\theta\alpha} \|\varphi\|_2.
		\]
Using above estimates,  \eqref{eq:Gamma}, \eqref{eq:dif-Lp} and Sobolev embedding, we obtain
		\begin{equation}\label{eq:D-prod3}
			\begin{aligned}
				\|\Gamma_{\alpha}(\bar{u}, \Lambda^{-\alpha} \varphi)\|_2 & \lesssim \int_{\mR^d} \frac{\l\| \l(\bar{u}(\cdot+z)-\bar{u}(\cdot)\r) \l(\Lambda^{-\alpha} \varphi(\cdot+z)-\Lambda^{-\alpha}\varphi(\cdot)\r) \r\|_2}{|z|^{d+\alpha}} \ \d z\\
				& \lesssim \|\bar{u}\|_{C^\delta} \|\varphi\|_2 \int_{|z|\leq 1} |z|^{-d+\delta} \, \d z + \|\bar{u}\|_{\frac{2p}{2-p}} \|\varphi\|_2 \int_{|z|>1} |z|^{-d-\alpha+\theta \alpha}\, \d z\\
				& \lesssim \|\bar{u}\|_{H^{\alpha, \frac{2p}{2-p}}} \|\varphi\|_2 \\
				&\lesssim\|\bar{\omega}\|_{H^{\alpha, p}}  \|\varphi\|_2 .
			\end{aligned}
		\end{equation}
		Combining \eqref{eq:prod-H}--\eqref{eq:D-prod3}, we arrive at
		$$
		    \|\bar{u}_s (\omega_s-\bar{\omega}_s)\|_{\dot{H}^{-\alpha}}\lesssim \|\omega_s-\bar{\omega}_s\|_{\dot{H}^{-\alpha}} \|\bar{\omega}_s\|_{H^{\alpha, p}}.
		$$

		Plugging the above estimate into \eqref{eq:aprox2}, using H\"older's inequality and noting that $q'<2$, we obtain
		\begin{equation*}
			\begin{aligned}
				\|I^2_t\|_{\dot{H}^{-\alpha}} &\lesssim \l[\int_0^t \frac1{(t-s)^{q'/2}}\,\d s \r]^{1/q'} \l[ \int_0^t \|\bar{\omega}_s \|_{H^{\alpha,p}}^q \|\omega_s-\bar\omega_s \|_{\dot{H}^{-\alpha}}^q \,\d s \r]^{1/q}\\
				&\lesssim_{T} \l[ \int_0^t   \|\omega_s-\bar\omega_s \|_{\dot{H}^{-\alpha}}^q \, \|\bar{\omega}_s \|_{H^{\alpha,p}}^q  \d s \r]^{1/q}.
			\end{aligned}
		\end{equation*}
		This together with \eqref{proof-difference} and \eqref{eq:aprox1} implies
		$$\|\omega_t -\bar\omega_t\|_{\dot H^{-\alpha}}^q \lesssim_{\alpha, p, q,T} \int_0^t   \|\omega_s-\bar{\omega}_s \|_{\dot H^{-\alpha}}^q \l( \|\omega_0 \|_{p}^q  + \|\bar{\omega}_s \|_{H^{\alpha,p}}^q  \r)\d s + \|Z_t \|_{\dot H^{-\alpha}}^q. $$
		Gronwall's inequality yields
		\begin{equation}
			\begin{aligned}
				\|\omega-\bar{\omega}\|_{C_T \dot  H^{-\alpha}}^q & \leq C \|Z \|_{C_T \dot H^{-\alpha}}^q \exp \l\{ C  {\|\omega_0 \|_{p}^q} +C \|\bar{\omega} \|_{L^q_TH^{\alpha,p}}^q\r\}\\
				& \leq C \|Z \|_{C_T \dot H^{-\alpha}}^q \exp \l\{ C  {\|\omega_0 \|_{p}^q}  \r\},
			\end{aligned}
		\end{equation}
		where $C$ is a constant depending only on $\alpha, p, q,T$.
		Note that the exponential is deterministic, thus we complete the proof by applying the maximal estimate in Corollary \ref{cor-stoch-convol-2} (as $d=2$, the range for $\alpha$ becomes $(2-p, 1]$ which contains the one in the statement of Theorem \ref{thm:SEE}).
	\end{proof}

%	\appendix
%	
%	\section{}
%	\setcounter{equation}{0}
%	\renewcommand\theequation{A.\arabic{equation}}

	\subsection*{Acknowledgements}
The first author is grateful to the National Key R\&D Program
of China (No. 2020YFA0712700), the National Natural Science Foundation of China (Nos. 11931004, 12090010, 12090014) and the Youth Innovation Promotion Association, CAS (Y2021002). The second author is supported by JSPS KAKENHI, Grant-in-Aid for Scientific Research (C) (No. 20K0367). The research of Guohuan Zhao is supported by the National Natural Science Foundation of China (Nos. 12288201, 12271352, 12201611).

\end{document}